\documentclass[11pt]{article}
\usepackage{amsmath, amsthm, amssymb, mathtools, url, hyperref,
  subfig, mathtools, graphicx}
\usepackage[noadjust]{cite}
\usepackage[utf8]{inputenc}
\usepackage{soul}
\usepackage{xcolor}
\usepackage{xypic}
\usepackage[margin=1in]{geometry}

\swapnumbers

\theoremstyle{plain}
\newtheorem{theorem}{\hspace{0em}Theorem}[section]
\newtheorem*{theorem*}{Theorem}
\newtheorem{corollary}[theorem]{\hspace{0em}Corollary}

\newtheorem*{question*}{Question}

\newtheorem{lemma}[theorem]{\hspace{0em}Lemma}
\newtheorem{proposition}[theorem]{\hspace{0em}Proposition}
\theoremstyle{definition}
\newtheorem{definition}[theorem]{\hspace{0em}Definition}
\newtheorem*{definition*}{Definition}
\newtheorem{remark}[theorem]{\hspace{0em}Remark}

\newtheorem{notation}[theorem]{\hspace{0em}Notation}

\newtheorem*{notation*}{Notation}

\newtheoremstyle{myclaim}
  {1ex}
  {1ex}
  {\it}
  {\parindent}
  {\it}
  {.}
  { }
  {}
\theoremstyle{myclaim}

\newtheorem*{claim*}{Claim}

\newtheoremstyle{note}
  {}
  {}
  {}
  {}
  {\normalfont}
  {.}
  {.5em}
  {}

\theoremstyle{note}

\newcommand{\Z}{\ensuremath{\mathbb{Z}}}

\newcommand{\C}{\ensuremath{\mathbb{C}}}
\newcommand{\T}{\ensuremath{\mathbb{T}}}
\newcommand{\F}{\ensuremath{\mathbb{F}}}

\DeclareMathOperator{\id}{id}
\DeclareMathOperator{\GL}{GL}           
       
\DeclareMathOperator{\Hom}{Hom}

\numberwithin{equation}{section}

\makeatletter
\newcommand*{\centerfloat}{%
  \parindent \z@
  \leftskip \z@ \@plus 1fil \@minus \textwidth
  \rightskip\leftskip
  \parfillskip \z@skip}
\makeatother

\let\originalleft\left
\let\originalright\right
\renewcommand{\left}{\mathopen{}\mathclose\bgroup\originalleft}
\renewcommand{\right}{\aftergroup\egroup\originalright}

\DeclareSymbolFont{EUEX}{U}{euex}{m}{n}

\DeclareSymbolFont{euexlargesymbols}{U}{euex}{m}{n}
\DeclareMathSymbol{\intop}{\mathop}{euexlargesymbols}{"52}
     \def\int{\intop\nolimits}

\DeclareSymbolFont{euexsymbols}     {U}{euex}{m}{n}
\DeclareMathSymbol{\smallint}{\mathop}{euexsymbols}{"52}

\newcommand{\simplex}[1]{\ensuremath{\langle #1 \rangle}}
\newcommand{\bary}[1]{\ensuremath{ \simplex{ #1 }\,\hat{}\, }}
\newcommand{\baryv}[1]{\ensuremath{ \mathring{v}_{#1} }}
\DeclareMathOperator{\dist}{dist}
\DeclareMathOperator{\U}{U}
\DeclareMathOperator{\Ext}{Ext}
\newcommand{\tu}{{U}^{\nu}}
\newcommand{\tv}{\widetilde{v}}
\newcommand{\ep}{\varepsilon}
\newcommand{\bpi}{\breve{\pi}}

\newcommand{\MR}[1]{%
  \href{http://www.ams.org/mathscinet-getitem?mr=#1}{MR #1}
}

\begin{document}

\title{%
  Almost flat K-theory of classifying spaces}

\author{%
  Jos\'e R. Carri\'on%
  \thanks{Partially supported by NSF Postdoctoral Fellowship
    \#DMS-1303884 and the Center for Symmetry and Deformation at the
    University of Copenhagen.}\\
  Department of Mathematics\\
  Penn State University\\
  University Park, PA, 16803\\
  United States\\
  jcarrion@psu.edu
  \and
  Marius Dadarlat%
  \thanks{Partially supported by NSF grant \#DMS-1362824. 
  \newline
 2010 Mathematics Subject Classification 46L05 (primary), 46L80, 46L85 (secondary)}\\
  Department of Mathematics\\
  Purdue University\\
  West Lafayette, IN, 47907\\
  United States\\
  mdd@math.purdue.edu
}

\maketitle

\begin{abstract}
  We give a rigorous account and prove continuity properties for the
  correspondence between almost flat bundles on a triangularizable
  compact connected space and the quasi-representations of its
  fundamental group.  For a discrete countable group $\Gamma$ with
  finite classifying space $B\Gamma$, we study a correspondence
  between between almost flat K-theory classes on $B\Gamma$ and group
  homomorphism $K_0(C^*(\Gamma))\to \mathbb{Z}$ that are implemented
  by pairs of discrete asymptotic homomorphisms
  from $C^*(\Gamma)$ to matrix algebras.
\end{abstract}

\section{Introduction}
\label{sec:introduction}
Connes, Gromov and Moscovici \cite{Connes-Gromov-etal90} developed and
used the concepts of almost flat bundle, almost flat K-theory class
and group quasi-representation as tools for proving the Novikov
conjecture for large classes of groups.  For a compact manifold $M,$
it was shown in \cite{Connes-Gromov-etal90} that the signature with
coefficients in a (sufficiently) almost flat bundle is a homotopy
invariant. Moreover, the authors indicate that they have a
reformulation of the notion of almost flatness to bundles with
infinite dimensional fibers which allows them to show that if $\Gamma$
is a countable discrete group such that all classes of $K^0(B\Gamma)$
are almost flat (up to torsion), then $\Gamma$ satisfies the Novikov
conjecture \cite[Sec. 6]{Connes-Gromov-etal90}.

The problem of constructing nontrivial almost flat K-theory classes is
interesting in itself.  Suppose that the classifying space $B\Gamma$
of a countable discrete group $\Gamma$ admits a realization as a
finite simplicial complex.  Using results of Kasparov
\cite{Kasparov88}, Yu \cite{Yu00} and Tu \cite{Tu05}, the second-named
author showed in \cite{Dadarlat14} that if $\Gamma$ is coarsely
embeddable in a Hilbert space and the full group C*-algebra
$C^*(\Gamma)$ is quasidiagonal, then all classes in $K^0(B\Gamma)$ are
almost flat.

Inspired by \cite{Connes-Gromov-etal90}, in this paper we investigate
the correspondence between between the almost flat classes in
$K^0(B\Gamma)$ and the group homomorphisms $h\colon
K_0(C^*(\Gamma))\to \Z$ that are implemented by pairs of discrete
asymptotic homomorphisms $\{\pi^\pm_n\colon C^*(\Gamma)\to
M_{k(n)}(\C)\}_{n=1}^\infty$, in the sense that $h(x)\equiv
(\pi^{+}_n)_\sharp(x)-(\pi^{-}_n)_\sharp(x)$ for $x\in
K_0(C^*(\Gamma))$; see Definition~\ref{def:pushforward}.  It turns out
that this correspondence is one-to-one (modulo torsion) if the full
assembly map $\mu\colon K^0(B\Gamma) \to K_0(C^*(\Gamma))$ is
bijective.  See Theorem~\ref{thm_cor:almost-flat-surj} and its
generalization with coefficients
Theorem~\ref{thm:main-correspondence}.

We take this opportunity to give a self-contained presentation of the
correspondence (and its continuity properties) between almost flat
bundles on a connected triangularizable compact space and the
quasi-representations of its fundamental group; see
Theorems~\ref{thm:main-correspondence}, \ref{thm:approx-inverse}. As
far as we can tell, while this correspondence was more or less known
to the experts, it has not been well documented in the literature.  We
rely on work of Phillips and Stone \cite{Phillips-Stone86,
  Phillips-Stone90}. These authors studied topological invariants
associated with lattice gauge fields via a construction which
associates an almost flat bundle to a lattice gauge field with
controlled distortion and small modulus of continuity.

The following terminology is useful for further discussion of our
results.  If $\mathbf{k}=(k(n))_{n=1}^\infty$ is a sequence of natural
numbers, we write $Q_{\mathbf{k}}=\prod_{n=1}^\infty M_{k(n)}(\C) /
\sum_{n=1}^\infty M_{k(n)}(\C)$. Recall that a separable C*-algebra
$A$ is MF is it embeds as a C*-subalgebra of $Q_{\mathbf{k}}$ for some
$\mathbf{k}$.  In other words, there are sufficiently many
$*$-homomorphisms $A \to Q_{\mathbf{k}}$ to capture the norm of the
elements in $A$, \cite{Blackadar-Kirchberg97}.  By analogy, let us say
that a separable C*-algebra $A$ is ``K-theoretically MF'' if there
exist sufficiently many $*$-homomorphisms $A \to Q_{\mathbf{k}}$ to
capture the K-theory of $A$ in the following sense: For any
homomorphism $h\colon K_0(C^*(\Gamma))\to \Z$ there exist $\mathbf{k}$
and two $*$-homomorphisms $\pi^\pm\colon A \to Q_{\mathbf{k}}$ such
that $\pi^{+}_*(x)-\pi^{-}_*(x)=h_\infty(x)$ for all $x\in
K_0(A)$. Here we identify $K_0(Q_{\mathbf{k}})$ with a subgroup of
$\prod_{n=1}^\infty \Z / \sum_{n=1}^\infty \Z$, and $h_\infty(x)$ is
the coset of element $(h(x),h(x),h(x),\dots)\in \prod_{n=1}^\infty\Z$.
With this terminology, Theorem~\ref{thm_cor:almost-flat-surj} reads as
follows:
\begin{theorem*}
  Let $\Gamma$ be a discrete countable group whose classifying space
  $B\Gamma$ is a finite simplicial complex. If the full assembly map
  $\mu\colon K_0(B\Gamma)\to K_0(C^*(\Gamma))$ is bijective, then the
  following conditions are equivalent:
  \begin{enumerate}
  \item All elements of $K^0(B\Gamma)$ are almost flat modulo torsion;
  \item $C^*(\Gamma)$ is  K-theoretically MF.
  \end{enumerate}
\end{theorem*}
By a result of Higson and Kasparov \cite{Higson-Kasparov01}, the
assumptions of this theorem are satisfied by groups $\Gamma$ with the
Haagerup property and finite classifying spaces.

We note that a separable quasidiagonal C*-algebra that satisfies the
UCT is K-theoretically MF by \cite[Prop. 2.5]{Dadarlat14}. We suspect
that a similar result holds for MF algebras satisfying the UCT.

\section{Basic definitions and notation}
\label{sec:basic-def-not}
Let $X$ be a connected compact metric space that admits a finite
triangulation. This means that $X$ is the geometric realization of
some connected finite simplicial complex $\Lambda$, written
$X=|\Lambda|$.  Let $\Gamma = \pi_1(X)$ be the fundamental group of
$X$. Let $A$ be a unital $C^*$-algebra.  We establish a correspondence
between quasi-representations of $\Gamma$ into $\GL(A)$ and almost
flat bundles over $X$ with fiber $A$ and structure group $\GL(A)$.  To
be more accurate, this correspondence is between almost unitary
quasi-representations and almost unitary principal bundles.  It is
convenient to work in a purely combinatorial context. Following
\cite{Phillips-Stone86}, we call the combinatorial version of an
almost flat bundle an \emph{almost flat coordinate bundle}.  See
Definition~\ref{def:af-bundle}.  We prove the equivalence of these
concepts in Proposition~\ref{prop:correspondence_af}.

For our combinatorial approach we begin with a finite simplicial
complex $\Lambda$ endowed with some additional structure, as follows.
The vertices of $\Lambda$ are denoted by $i,j,k$ with possible
indices. The set of $k$-simplices of $\Lambda$ is denoted by
$\Lambda^{(k)}.$ We fix a maximal tree $T \subset \Lambda$ and a root
vertex $i_0$.

\subsection*{The edge-path group and quasi-representations}
\label{sec:edge-path-group}

It will be convenient to view the fundamental group of the geometric
realization of $\Lambda$ as the \emph{edge-path group} $E(\Lambda,
i_0)$ of $\Lambda$.  The groups $\pi_1(|\Lambda|)$ and $E(\Lambda,
i_0)$ are isomorphic and we will simply write $\Gamma$ for either
group.

Once $\Lambda$ and $T$ are fixed so is the following standard
presentation of $\Gamma$ in terms of the edges of $\Lambda$ (see
e.g. \cite[Sections 3.6--3.7]{Spanier66}): $\Gamma$ is isomorphic to
the group generated by the collection of all edges $\simplex{ i, j }$
of $\Lambda$ subject to the following relations:

\begin{itemize}
\item if $\simplex{ i, j }$ is an edge of $T$, then $\simplex{ i,
    j } = 1$;
\item if $i,j,k$ are vertices of a simplex of $\Lambda$, then
  $\simplex{ i, j }\simplex{ j, k } = \simplex{ i, k }$.
\end{itemize}
We should point out that in the second relation $i$, $j$, and $k$ are
not necessarily distinct; consequently $\simplex{ i, j}\simplex{j, i}
= \simplex{i , i}$.  Since $\simplex{i, i} = \simplex{i}$ belongs to
$T$ one has $\simplex{i, j}^{-1} = \simplex{j, i}$ as expected.
\begin{notation}
  \label{not:edge-path-grp}
  Let $\F_\Lambda$ be the free group generated by the edge set of
  $\Lambda$ and let $q\colon \F_\Lambda\to \Gamma$ be the group
  epimorphism corresponding the presentation of $\Gamma$ just
  described.

  Write $\mathcal{F}_\Lambda$ for the image under $q$ of the edge set
  of $\Lambda$; this is a symmetric generating set for $\Gamma$.
  Write $\gamma_{ij} = q( \simplex{i,j} )$ for the elements of
  $\mathcal{F}_\Lambda$.  Let $\mathcal{R}\subset \mathbb{F}_\Lambda$
  be the collection of all relators:
  \begin{displaymath}
    \mathcal{R} :=
    \{ \simplex{i, j} \mid \simplex{i, j} \text{ is an
      edge of } T \}
    \cup
    \{ \simplex{ i, j } \simplex{ j, k }
    \simplex{ i, k }^{-1} \mid i, j, k\text{ are vertices
      of a simplex of }\Lambda \}.
  \end{displaymath}

  Choose a set-theoretic section $s\colon \Gamma\to \F_\Lambda$ of $q$
  that takes the neutral element of $\Gamma$ to the neutral element of
  $\mathbb{F}_\Lambda$.  This section $s$ will remain fixed for the
  rest of the paper.
\end{notation}

We introduce one last notation before the definition. If $A$ is a
unital C*-algebra and $\delta>0$, set
\[
\U(A)_\delta = \{ v\in A : \dist(v, \U(A)) < \delta \}.
\]
Note that $\U(A)_\delta \subseteq \GL(A)$ if $\delta<1$.

\begin{definition}\mbox{}
  \label{def:q-rep}
  Let $A$ be a unital C*-algebra.
  \begin{enumerate}
  \item Let $\mathcal{F} \subset \Gamma$ be finite and let
    $0<\delta<1$.  A function $\pi\colon \Gamma\to \GL(A)$ is an
    \emph{$(\mathcal{F}, \delta)$-representation} of $\Gamma$ if
    \begin{enumerate}
    \item $\pi(\gamma) \in \U(A)_\delta$ for all
      $\gamma\in\mathcal{F}$;
    \item $\| \pi(\gamma\gamma') - \pi(\gamma)\pi(\gamma') \| <
      \delta$ for all $\gamma, \gamma'\in\mathcal{F}$;
    \item $\pi(e) = 1_A$, where $e$ is the neutral element of $\Gamma.$
    \end{enumerate}

  \item Define a pseudometric $d$ on the set of all bounded maps
    $\Gamma\to A$ by
    \[
    d( \pi, \pi' ) = \max_{\gamma\in \mathcal{F}_\Lambda}
    \| \pi(\gamma) - \pi'(\gamma) \|.
    \]
  \end{enumerate}
\end{definition}

We may sometimes refer to an $(\mathcal{F}, \delta)$-representation as
a ``quasi-representation'' without specifying $\mathcal{F}$ or
$\delta$. In most of the paper we will take
$\mathcal{F}=\mathcal{F}_{\Lambda}$.

\subsection*{The dual cover and almost flat coordinate bundles}

Once $\Lambda$ is fixed, so is a cover $\mathcal{C}_\Lambda$ of
$|\Lambda|$, called the dual cover.  We recall its definition
(borrowing heavily from the appendix of \cite{Phillips-Stone90}).

\begin{definition}
  Let $\sigma = \simplex{0, \dots, r}$ be a simplex of $\Lambda$.  For
  $i\in \{0, \dots, r\}$, the \emph{dual cell block} $c_i^\sigma$,
  dual to $i$ in $\sigma$, is defined in terms of the barycentric
  coordinates $(t_o, \dots, t_r)$ by
  \[
  c_i^\sigma =
  \{ (t_0, \dots, t_r) \mid t_i \geq t_j\text{ for all } j\}
  \subset |\Lambda|.
  \]
  The \emph{dual cell} $c_i$, dual to the vertex $i$, is the union of
  cell blocks dual to $i$:
  \[
  c_i = \cup\{ c_i^\sigma \mid i\in \sigma \}.
  \]
  The \emph{dual cover} $\mathcal{C}_\Lambda$ is the collection of all
  dual cells.  (See Figure~\ref{fig:dual-cells}.)
\end{definition}

\begin{figure}
  \centering
  \subfloat[]{\includegraphics[width=.3\textwidth]%
    {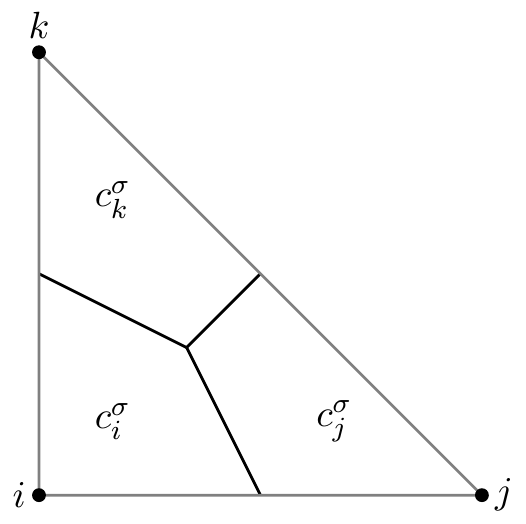}}\hspace{.1\textwidth}%
  \subfloat[]{\includegraphics[width=.4\textwidth]%
    {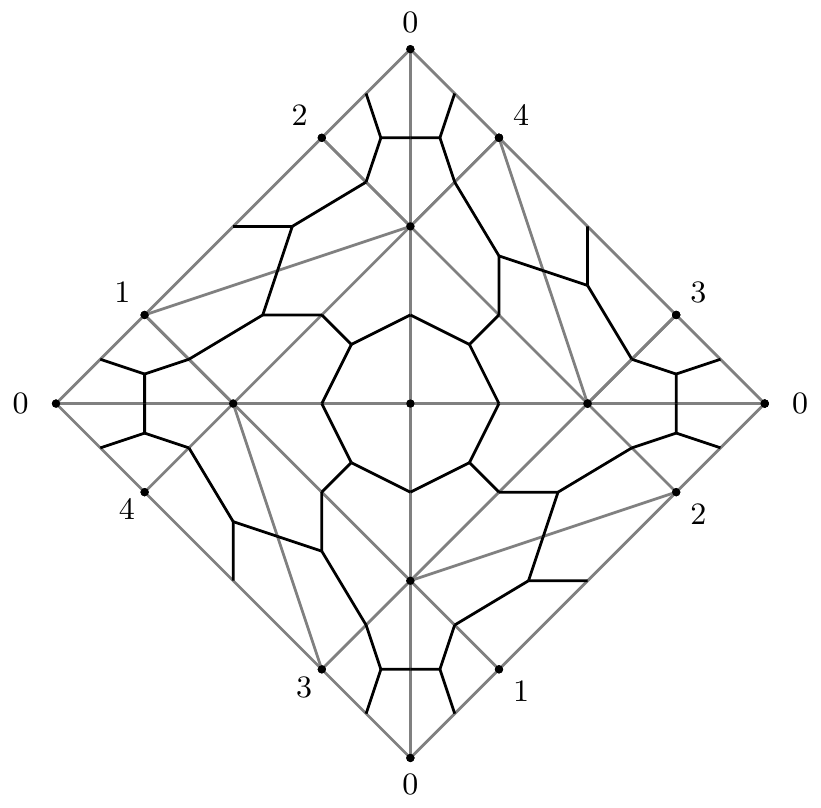}}
  \caption{(a) Dual cell blocks in a simplex $\sigma = \langle i, j, k
    \rangle$.
    (b) A triangulation of $\T^2$ with the dual cell structure
    highlighted.}
  \label{fig:dual-cells}
\end{figure}

\begin{notation}
  We usually write $c_{ij}$ for the intersection $c_i \cap c_j$,
  $c_{ijk}$ for $c_i\cap c_j\cap c_k$ etc.  The barycenter of a
  simplex $\sigma$ is denoted $\hat{\sigma}$.  Note that $\bary{ i, j
  } \in c_{ij}$.
\end{notation}

\begin{definition}
  \label{def:af-bundle}
  Recall that we have fixed a unital $C^*$-algebra $A$.
  \begin{enumerate}
  \item An \emph{$\varepsilon$-flat $\GL(A)$-coordinate bundle on
      $\Lambda$} is a collection of continuous functions $\mathbf{v} =
    \{ v_{ij}\colon c_{ij}\to \GL(A) \mid \simplex{i,j} \in
    \Lambda^{(1)} \}$ satisfying:
    \begin{enumerate}
    \item $v_{ij}(x) \in \U(A)_\varepsilon$ for all $x\in c_{ij}$ and all
      $\simplex{ i, j } \in \Lambda^{(1)}$;
    \item $v_{ij}(x) = v_{ji}(x)^{-1}$ for all $x\in c_{ij}$ and all
      $\simplex{ i, j } \in \Lambda^{(1)}$;
    \item $v_{ik}(x) = v_{ij}(x) v_{jk}(x)$ for all $x\in c_{ijk}$ and
      all $\simplex{ i, j, k } \in \Lambda^{(2)}$; and
    \item $\| v_{ij}(x) - v_{ij}(y) \| < \varepsilon$ for all $x, y \in
      c_{ij}$ and all $\simplex{ i, j }\in \Lambda^{(1)}$.
    \end{enumerate}

  \item Define a metric $d$ on the set of all
    $\GL(A)$-coordinate bundles on $\Lambda$ by
    \begin{displaymath}
      d( \mathbf{v}, \mathbf{v}' ) =
      \max_{ \simplex{i, j}\in \Lambda }
      \max_{ x\in c_{ij} }
      \| v_{ij}(x) - v'_{ij}(x) \|.
    \end{displaymath}
  \end{enumerate}
\end{definition}

We may sometimes refer to an $\varepsilon$-flat coordinate bundle as
an ``almost flat coordinate bundle'' without specifying $\varepsilon$.
We think of an almost flat coordinate bundle $\mathbf{v}$ as a
collection of transition functions defining a bundle over $|\Lambda|$,
with fiber $A$, that has ``small'' curvature.  We substantiate this
point of view in Section~\ref{sec:7} where we show that there are
positive numbers $\varepsilon_0,\nu, r$ that depend only on $\Lambda$
such that for any $\varepsilon$-flat $\GL(A)$-coordinate bundle
$v_{ij}$ on $\Lambda$, with $\ep<\ep_0$, there is a 1-{\v C}ech
cocycle $\tv_{ij}\colon V_i\cap V_j\to U(A)_{r\ep}$ that extends
$v_{ij}$ to prescribed open sets and $\tv_{ij}$ is $r\varepsilon$-flat
in the sense that $\|\tv_{ij}(x)-\tv_{ij}(x')\|<r\varepsilon$ for all
$x \in V_i\cap V_j$.  Here $V_i=\{x\in |\Lambda| \colon
\dist(x,c_i)<\nu\}$.  Since the sets $V_i$ are open, the usual gluing
construction based on $\tv_{ij}$ defines a locally trivial (almost
flat) bundle.  These objects are closely related to almost flat
$K$-theory classes; see Section~\ref{sec:an-appl} for details.

\section{The correspondence between almost flat bundles and
  quasi-representations}
\label{sec:results}
We state the main results on this topic. The proofs are given in
subsequent sections.

\begin{theorem}
  \label{thm:main-correspondence}
  Let $\Lambda$ be a finite connected simplicial complex with
  fundamental group $\Gamma$.  There exist positive numbers $C_0$,
  $\delta_0$, and $\varepsilon_0$ such that the following holds.

  If $A$ is a unital $C^*$-algebra, then there are functions
  \[
  \left\{
    \begin{array}{c}
      \varepsilon_0\text{-flat }\GL(A)\\
      \text{coordinate bundles on } \Lambda
    \end{array}
  \right\}
  \begin{array}{c}
    \alpha\\[-5pt]
    \xrightarrow{\hspace{20pt}}\\
    \xleftarrow{\hspace{20pt}}\\[-5pt]
    \beta
  \end{array}
  \left\{
    \begin{array}{c}
      (\mathcal{F}_\Lambda, \delta_0)\text{-representations}\\
      \text{of } \Gamma \text{ to } \GL(A)
    \end{array}
  \right\}
  \]
  such that:
  \begin{enumerate}
  \item if $0 < \varepsilon < \varepsilon_0$ and $\mathbf{v}$ is an
    $\varepsilon$-flat $\GL(A)$-coordinate bundle on $\Lambda$, then
    $\alpha(\mathbf{v})$ is an $(\mathcal{F}_\Lambda, C_0
    \varepsilon)$-representation of $\Gamma$ to $\GL(A)$; and
  \item if $0 < \delta < \delta_0$ and $\pi\colon \Gamma\to \GL(A)$ is
    an $(\mathcal{F}_\Lambda, \delta)$-representation, then
    $\beta(\pi)$ is a $C_0 \delta$-flat $\GL(A)$-coordinate bundle on
    $\Lambda$.
  \end{enumerate}
  Moreover:
  \begin{enumerate}
    \setcounter{enumi}{2}
  \item if $0 < \varepsilon < \varepsilon_0$ and $\mathbf{v}$ and
    $\mathbf{v'}$ are $\varepsilon$-flat $\GL(A)$-coordinate bundles,
    then $d( \alpha(\mathbf{v}), \alpha(\mathbf{v'}) ) < d(\mathbf{v},
    \mathbf{v'})+C_0\varepsilon $.
  \item if $0 < \delta < \delta_0$ and $\pi, \pi'\colon \Gamma\to
    \GL(A)$ are $(\mathcal{F}_\Lambda, \delta)$-representations, then
    $d(\beta(\pi), \beta(\pi')) < d(\pi, \pi')+C_0\delta$.
  \end{enumerate}
\end{theorem}
\begin{definition}
  \label{def:normalized-coord-bundle}
  A $\GL(A)$-coordinate bundle $\mathbf{v} = \{ v_{ij} \}$ is
  \emph{normalized} if $v_{ij}( \bary{i,j} ) = 1_A$ for every edge
  $\simplex{i, j}$ of $T$.
\end{definition}

By Proposition~\ref{prop:T-normalized}, if a vector bundle over
$|\Lambda|$ may be represented by an $\varepsilon$-flat
$\GL(A)$-coordinate bundle, then it can be represented by a normalized
$C\varepsilon$-flat $\GL(A)$-coordinate bundle (where $C > 0$ depends
only on $\Lambda$).

\begin{theorem}
  \label{thm:approx-inverse}
  Let $\Lambda$ be a finite connected simplicial complex with
  fundamental group $\Gamma$.  There exist positive numbers $C_1,
  \varepsilon_1$ and $\delta_1$ such that the following holds for any
  unital $C^*$-algebra $A$.
  \begin{enumerate}
  \item If $0 < \varepsilon < \varepsilon_1$ and $\mathbf{v}$ is a
    normalized $\varepsilon$-flat $\GL(A)$-coordinate bundle on
    $\Lambda$, then
    \[
    d\big( (\beta\circ \alpha)(\mathbf{v}), \mathbf{v} \big)
    \leq C_1\varepsilon.
    \]

  \item If $0 < \delta < \delta_1$ and $\pi\colon \Gamma\to \GL(A)$ is
    an $(\mathcal{F}_\Lambda, \delta)$-representation, then
    \[
    d\big( (\alpha\circ \beta)(\pi), \pi \big) \leq C_1\delta.
    \]
  \end{enumerate}
\end{theorem}

\section{From almost flat bundles to quasi-representations}
\label{sec:bundle-to-qrep}
In this section we construct the map $\alpha$ announced in
Section~\ref{sec:results}.  It is a combinatorial version of a
construction due to Connes-Gromov-Moscovici
\cite{Connes-Gromov-etal90} involving parallel transport on a smooth
manifold.

 Let $\mathbf{v} = \{ v_{ij}\colon c_{ij}\to \GL(A) \}$ be
an $\varepsilon$-flat coordinate bundle on $\Lambda$.  We
will define a quasi-representation $\alpha(\mathbf{v}) \colon
\Gamma\to \GL(A)$ with properties described in
Proposition~\ref{prop:qrep-existence}.
\begin{notation}
  \label{not:baryv}
  Recall that the barycenter of a 1-simplex $\simplex{i, j} \in
  \Lambda^{(1)}$ is written $\bary{i, j}$.  For such a 1-simplex, let
  $\baryv{ij} = v_{ij}( \bary{i, j} ) \in \GL(A)$.
For a path $I = ( i_1, \dots, i_m )$ of vertices in $\Lambda$, let
  \begin{align*}
    \baryv{I} &= \baryv{i_1 i_2}  \dots \baryv{i_{m-1}i_m}.
  \end{align*}
\end{notation}

\begin{definition}
  \label{def:quasi-rep-constr}
  Define a group homomorphism
  $\tilde{\pi}=\tilde{\pi}_{\mathbf{v}}\colon \F_\Lambda\to \GL(A)$ as
  follows.  If $\simplex{ i, j } \in \Lambda^{(1)}$, let $I = (i_0,
  \dots, i)$ be the unique path along $T$ from $i_0$ to $i$ and $J =
  (i_0, \dots, j)$ be the unique path from $i_0$ to $j$.  Set
  \begin{equation}
    \label{eq:1}
    \tilde{\pi}\big( \simplex{i, j} \big) =
    \baryv{I} \baryv{ij} \baryv{J}^{-1}.
  \end{equation}
  Finally, set
  \[
  \alpha(\mathbf{v})
  = \tilde{\pi}\circ s\colon \Gamma\to \GL(A),
  \]
  where $s$ is the set theoretic section of $q\colon \F_\Lambda\to
  \Gamma$ that was fixed in Notation~\ref{not:edge-path-grp}.
\end{definition}

\begin{lemma}
  \label{lem:U-epsilon-prod} Let $\nu>0$ and $0<\varepsilon<1$.  If
  $x_1, \dots, x_m\in A$, $u_1, \dots, u_m\in \U(A)$ and
  $\|x_i-u_i\|<\nu\varepsilon$ for all $i$, then $\| x_1\dots x_m -
  u_1\dots u_m \|<(1+\nu)^m\varepsilon$.  In particular, if $x_1,
  \dots, x_m\in \U(A)_\varepsilon$, then $x_1 \dots x_m \in
  \U(A)_{2^m\varepsilon}$.
\end{lemma}

\begin{proof}
  For $i \in \{1, \dots, m\}$, let $u_i\in \U(A)$ be such that $\| x_i
  - u_i \| < \varepsilon$.  One checks that
  \begin{equation*}
    \| x_1\dots x_m - u_1\dots u_m \|
    <
    \sum_{i=1}^m \| x_i - u_i \| (1 + \nu\varepsilon)^{m-i}
    < (1 + \nu)^m\varepsilon.\qedhere
  \end{equation*}
\end{proof}

\begin{notation}
  \label{not:ell-and-L}
  For $g \in \F_\Lambda$, let $\ell(g)\in \Z_{\geq 0}$ be the word
  length of $g$ with respect to the generating set $\Lambda^{(1)}$.
  We denote by $L$ the length (number of edges) of a longest path in
  $\Lambda$ that starts at the root $i_0$ and does not repeat any
  edge.
\end{notation}

\begin{lemma}
  \label{lem:pi-tilde-U-epsilon}
  If $g\in \F_\Lambda$, then $\dist( \tilde{\pi}( g ), \U(A) ) <
  2^{ 3L + \ell(g)} \varepsilon$.
\end{lemma}

\begin{proof}
  Equation~\eqref{eq:1} implies that $\tilde{\pi}( \simplex{i, j} )$
  is a product of at most $3L$ elements of $\U(A)_\varepsilon$ for any
  edge $\simplex{ i, j }$ of $\Lambda$.  It follows from
  Lemma~\ref{lem:U-epsilon-prod} that $\dist( \tilde{\pi}( \simplex{i,
    j} ), \U(A) ) < 2^{3 L}\varepsilon$.  Another application of
  Lemma~\ref{lem:U-epsilon-prod} ends the proof.
\end{proof}

\begin{proposition}
  \label{prop:qrep-existence}
  There is a constant $C_0'>0$, depending only on $\Lambda$, $T$,
  $i_0$, and $s$, such that if $\mathbf{v} $ is an $\varepsilon$-flat
  $\GL(A)$-coordinate bundle on $\Lambda$, then $\alpha(\mathbf{v})$
  is an $(\mathcal{F}_\Lambda, C_0'\varepsilon)$-representation of
  $\Gamma$ on $\GL(A)$.
\end{proposition}

\begin{proof}
  Write $\pi:=\alpha(\mathbf{v})$.  First we define a few constants so
  the proof will run more smoothly.

  Let $\ell_0 = \max \{ \ell( s(\gamma) ) : \gamma\in
  \mathcal{F}_\Lambda \cup \mathcal{F}_\Lambda\cdot
  \mathcal{F}_\Lambda \}$.

  If $\gamma, \gamma' \in \Gamma$, then
  $s(\gamma)s(\gamma')s(\gamma\gamma')^{-1}$ belongs to the kernel of
  $q\colon \F_\Lambda\to \Gamma$, that is, to the normal subgroup
  generated by the set of relators $\mathcal{R}$.  (See
  Notation~\ref{not:edge-path-grp}.)  For each pair $\gamma, \gamma'
  \in \mathcal{F}_\Lambda$ choose and fix a representation
  \begin{equation}\label{eqn:normal}
  s(\gamma)s(\gamma')s(\gamma\gamma')^{-1} =
  \prod_{n=1}^{m(\gamma,\gamma')} x_n r_n x_n^{-1}
  \end{equation}
  with $\{ x_n \} \subset \F_\Lambda$ and $\{ r_n \} \subset
  \mathcal{R}$.  Let $m=\max \{ m(\gamma, \gamma') : \gamma,
  \gamma'\in \mathcal{F}_\Lambda \}$ and let $\ell_1$ be the maximum
  of the lengths $\ell(x_n)$ of all the elements $x_n$ that appear in
  equation \eqref{eqn:normal} for all pairs $\gamma, \gamma'\in
  \mathcal{F}_\Lambda $.  Finally, let $C_0' = 2^{(15
    L+2\ell_1)m}\cdot 2^{4L+\ell_0}$.

  For any $\gamma, \gamma'\in \mathcal{F}_\Lambda$, we show that $\|
  \pi(\gamma)\pi(\gamma') - \pi(\gamma\gamma') \| < C_0'\varepsilon$.
  Using Lemma~\ref{lem:pi-tilde-U-epsilon} we note that $\|
  \tilde{\pi}( s(\gamma\gamma') ) \| < 1 + 2^{ 3 L +
    \ell_0}\varepsilon< 2^{ 4 L + \ell_0}$ and hence
  \begin{align}\label{eq:2}
    \| \pi(\gamma)\pi(\gamma') - \pi(\gamma\gamma') \|
    &\leq
    \| \tilde{\pi}( s(\gamma)s(\gamma')s(\gamma\gamma')^{-1} ) - 1 \|
    \| \tilde{\pi}( s(\gamma\gamma') ) \| \\
    &\leq
     2^{ 4 t + \ell_0}
     \| \tilde{\pi}( s(\gamma)s(\gamma')s(\gamma\gamma')^{-1} ) - 1 \|
     .\nonumber
  \end{align}

  Now, $r_n\in \mathcal{R}$ implies that either $r_n = \simplex{ i,
    j}$ is an edge of $T$ in which case $\tilde{\pi}( r_n ) = 1$ or
  $r_n = \simplex{ i, j } \simplex{ j, k } \simplex{ i, k }^{-1}$ for
  some vertices $i, j, k$, and hence
  \[
  \tilde{\pi}( r_n )
  =
  \baryv{I} \baryv{ij} \baryv{J}^{-1}
  \cdot
  \baryv{J} \baryv{jk} \baryv{K}^{-1}
  \cdot
  \baryv{K} \baryv{ik}^{-1} \baryv{I}^{-1}
  =
  \baryv{I} \cdot
  \baryv{ij} \baryv{jk} \baryv{ik}^{-1} \cdot
  \baryv{I}^{-1}.
  \]
  Let $t$ be the barycenter $\bary{ i, j, k }$.  Since
  $v_{ij}$, $v_{jk}$, and $v_{ik}^{-1}$ are $\varepsilon$-constant
  with norm $\leq 1+\varepsilon$, we get that
  \begin{align*}
    \| \tilde{\pi}( r_n ) - 1 \|
    &\leq
    \| \baryv{I} \cdot
       \baryv{ij} \baryv{jk} \baryv{ik}^{-1} \cdot
       \baryv{I}^{-1}-
       \baryv{I} \cdot
       v_{ij}(t) v_{jk}(t) v_{ik}^{-1}(t)
       \cdot
       \baryv{I}^{-1}
    \|\\
    &< \| \baryv{I} \| \| \baryv{I}^{-1} \| (1+\varepsilon)^2 3\varepsilon
    \\
    &< 2^{ 2L + 4 }\varepsilon \leq 2^{ 6 L }\varepsilon.
  \end{align*}
  By Lemma~\ref{lem:pi-tilde-U-epsilon} $\| \tilde{\pi}( x_n ) \|, \|
  \tilde{\pi}( x_n^{-1} ) \|\leq 1+2^{ 3L + \ell(x_n)}\varepsilon< 2^{
    4L + \ell_1}$.  Therefore
  \[
  \|\tilde{\pi}(x_n)\tilde{\pi}( r_n) \tilde{\pi}(x_n^{-1})-1\|< \|
  \tilde{\pi}( x_n ) \| \| \tilde{\pi}( x_n^{-1} ) \| \| \tilde{\pi}(
  r_n ) - 1 \| \leq 2^{ 14 L +2\ell_1}\varepsilon.
  \]

  Because
  \[
  \tilde{\pi}( s(\gamma)s(\gamma')s(\gamma\gamma')^{-1} )=
  \prod_{n=1}^{m(\gamma,\gamma')} \tilde{\pi}(x_n)\tilde{\pi}( r_n)
  \tilde{\pi}(x_n^{-1}),
  \]
  applying Lemma~\ref{lem:pi-tilde-U-epsilon} again we get

  \begin{equation}
    \label{eq:3}
    \| \tilde{\pi}( s(\gamma)s(\gamma')s(\gamma\gamma')^{-1})
    - 1 \|
    <
    ( 1 + 2^{14L+2\ell_1} )^m\varepsilon
    \leq
    2^{(15L+2\ell_1)m}\varepsilon.
  \end{equation}

  Combined with \eqref{eq:2}, this proves that for all
  $\gamma,\gamma'\in \mathcal{F}_\Lambda$
  \begin{align*}
    \| \pi(\gamma)\pi(\gamma') - \pi(\gamma\gamma') \|< 2^{ 4 L +
      \ell_0}\cdot 2^{(15 L+2\ell_1)m}\varepsilon = C_0'\varepsilon.
  \end{align*}

  We must also prove that $\pi(\gamma) \in \U(A)_{C_0'\varepsilon}$ if
  $\gamma\in \mathcal{F}_\Lambda$.  This follows immediately from
  Lemma~\ref{lem:pi-tilde-U-epsilon} since
  $\pi(\gamma)=\tilde{\pi}(s(\gamma))$ and $\ell(\gamma) \leq \ell_0$.
\end{proof}

Let us single out an estimate from the proof for later use
(cf. \eqref{eq:3}).

\begin{lemma}
  \label{lem:bound-on-R}
  There exists $K > 0$, depending only on $\Lambda$, $T$, $i_0$, and
  $s$, such that the following holds.

  Suppose $\mathbf{v}$ is an $\varepsilon$-flat $\GL(A)$-coordinate
  bundle on $\Lambda$.  Let $\tilde{\pi}=\tilde{\pi}_{\mathbf{v}}$ be
  as in Definition~\ref{def:quasi-rep-constr}.  If $\simplex{i, j}\in
  \Lambda^{(1)}$, then $g_{ij} := s(\gamma_{ij}) \cdot \simplex{i,
    j}^{-1} \in \ker q$ and $\| \tilde{\pi}(g_{ij}) - 1_A\| <
  K\varepsilon.$
\end{lemma}
\begin{proposition}
  \label{prop:T-normalized}
  Suppose $\mathbf{v} = \{ v_{ij} \}$ is an $\varepsilon$-flat
  $\GL(A)$-coordinate bundle on $\Lambda$.  Then there exist a
  constant $C > 0$, depending only on $\Lambda$, and elements
  $\lambda_i\in \GL(A)$ such that the coordinate bundle $ \mathbf{w} =
  \{ w_{ij} \}$ defined by \( w_{ij} = \lambda_i v_{ij} \lambda_j^{-1}
  \) is normalized and $C\varepsilon$-flat; $\mathbf{v}$ and
  $\mathbf{w}$ yield isomorphic bundles on $|\Lambda|$.
\end{proposition}

\begin{proof}
  Recall that $L$ is the length of the longest path in $T$ that starts
  at the root $i_0$ and has no backtracking. We show that $C:=4^{L+1}$
  verifies the statement.  For each vertex $i$ of $\Lambda$, let $I =
  (i_0, \dots, i)$ be the unique path along $T$ from $i_0$ to $i$ and
  using notation as in \ref{not:baryv} and \ref{def:quasi-rep-constr}
  set $\lambda_i:=\baryv{I}$. Thus \( w_{ij} = \baryv{I}v_{ij}
  \baryv{J}^{-1}.  \)

  Then clearly $w_{ij}(\bary{i,j})=1_A$ for all $\simplex{i,j}\in
  T^{(1)}$. Moreover
  $\|\lambda_i\|,\|\lambda^{-1}_i\|<(1+\varepsilon)^L< 2^L$ since
  $\|\baryv{i,j}\|<1+\varepsilon$ by hypothesis.  If $\simplex{ i, j
  }$ is a 1-simplex of $\Lambda$ and $x, y\in c_{ij}$, then
  \begin{align*}
    \| w_{ij}(x) - w_{ij}(y) \|
    &=
    \| \lambda_i (v_{ij}(x) -
       v_{ij}(y) )\lambda_{j}^{-1} \|
    < 2^L \cdot \varepsilon  \cdot 2^L <C\varepsilon.
  \end{align*}
  Moreover, since $\lambda_i, \lambda_{j}^{-1}\in
  \U(A)_{2^L\varepsilon}$ by Lemma~\ref{lem:U-epsilon-prod} and
  $v_{ij}(x)\in \U(A)_\varepsilon$ it follows immediately that
  $w_{ij}(x)\in \U(A)_{4^{L+1}\varepsilon}$.  Therefore, $\mathbf{w}$
  is $ C\varepsilon$-flat.

  By \cite[Theorem~3.2]{Karoubi78} $\mathbf{v}$ and $\mathbf{w}$ yield
  isomorphic bundles on $|\Lambda|$.
\end{proof}

\section{From quasi-representations to almost flat bundles}
\label{sec:qrep-to-bundle}
In this section we describe the map $\beta$ announced in
Section~\ref{sec:results}.  The idea is to extend a
quasi-representation of $\Gamma$ to a quasi-representation of the
fundamental groupoid of $\Lambda$ and reinterpret the latter as a
lattice gauge field.  This enables us to invoke a construction of
Phillips and Stone \cite{Phillips-Stone86, Phillips-Stone90} that
associates an almost flat coordinate bundle to a lattice gauge field
with controlled distortion and small modulus of continuity. For the
sake of completeness we give a full account of this construction.

Let $0 < \delta < 1/140L$, with $L$ as in \ref{not:ell-and-L}, and let
an $(\mathcal{F}_\Lambda, \delta)$-representation $\pi$ of $\Gamma$ be
given.  We will define an almost flat coordinate bundle $\beta(\pi)$
on $\Lambda$.

The following notation will be used in the definition.

\begin{notation}
  We fix a partial order \textbf{o} on the vertices of $\Lambda$ such
  that the set of vertices of any simplex of $\Lambda$ is a totally
  ordered set under \textbf{o}.  One may always assume that such an
  order exists by passing to the first barycentric subdivision of
  $\Lambda$: if $\hat{\sigma}_1$ and $\hat{\sigma}_2$ are the
  barycenters of simplices $\sigma_1$ and $\sigma_2$ of $\Lambda$,
  define $\hat{\sigma}_1 < \hat{\sigma}_2$ if $\sigma_1$ is a face of
  $\sigma_2$ (cf. \cite{Phillips-Stone90}).

  When we write $\sigma = \simplex{i_1 , \dots, i_m}$ it is implicit
  that the vertices of $\sigma$ are written in increasing
  \textbf{o}-order.
\end{notation}

\begin{notation}
  \label{not:mod-bary-coords}
  Following \cite{Phillips-Stone90}, we re-parametrize the dual cell
  blocks $c_i^\sigma$ using ``modified barycentric coordinates''
  $(s_0, \dots, s_r)$.  These are defined in terms of the barycentric
  coordinates by $s_j = t_j / t_i$.  In these coordinates $c_i^\sigma$
  is identified with the cube
  \[
  \{ (s_0, \dots, s_i, \dots, s_r)
  \mid
  s_i = 1 \text{ and } 0\leq s_j \leq 1\text{ for all }j\not= i\}.
  \]
  See Figure~\ref{fig:barycentric-vs-modified}.
\end{notation}

\begin{figure}
  \centering
  \subfloat[]{\includegraphics[width=.35\textwidth]%
    {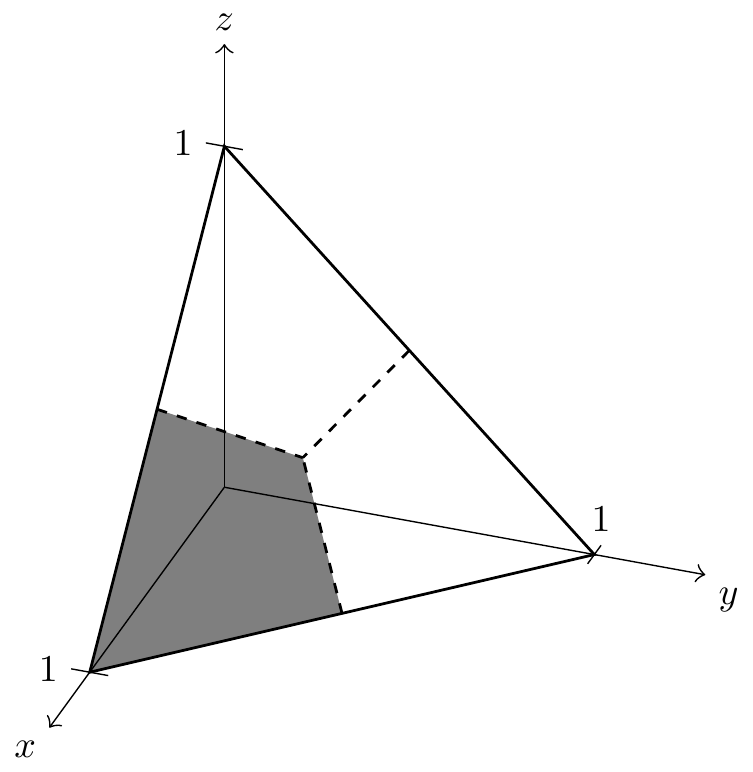}}\hspace{.15\textwidth}
  \subfloat[]{\includegraphics[width=.35\textwidth]%
    {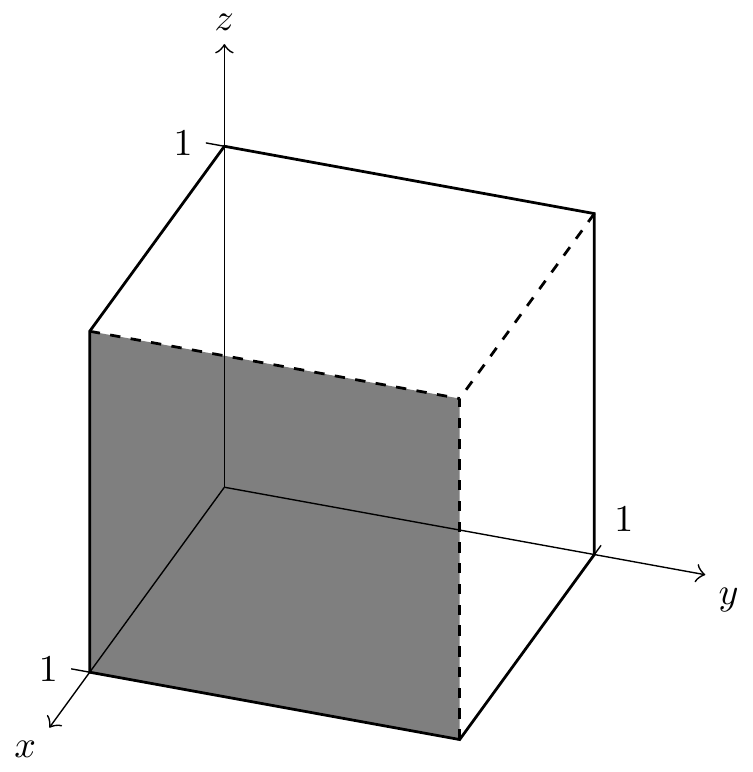}}
  \caption{(a) The identification of a dual cell block in a 2-simplex
    as given by barycentric coordinates.
    (b) The identification of a dual cell block in a 2-simplex
    as given by modified barycentric coordinates.}
  \label{fig:barycentric-vs-modified}
\end{figure}

The construction of $\beta(\pi)$ is inspired by (and borrows heavily
from) the work of Phillips and Stone \cite{Phillips-Stone90}.  It is
somewhat involved, but we outline the procedure in the following
definition before going into the details.

\begin{definition}
  \label{def:cocycle}
  Let $i$ and $j$ be adjacent vertices of $\Lambda$.  We will define
  $\beta(\pi) = \{ v_{ij}\colon c_{ij}\to \GL(A) \}$ by defining
  $v_{ij}$ on all the dual cell blocks $c_{ij}^\sigma$ such that
  $\sigma$ contains $i$ and $j$.

  \begin{enumerate}
  \item Let \( u_{ij}:=\bpi(\gamma_{ij}), \) where $\bpi\colon
    \Gamma\to \U(A)$ is the perturbation of $\pi$ provided by
    Proposition~\ref{prop:perturb-qrep}.
  \item Suppose $\sigma$ is a simplex in $\Lambda$ containing $i$ and
    $j$ and that $i < j$.  Write $\sigma$ (in increasing
    \textbf{o}-order) as $\sigma = \simplex{ 0, \dots, i, \dots, j,
      \dots, r}$.

    For an \textbf{o}-ordered subset of vertices $I = \{ i=i_1 < i_2 <
    i_3 < \dots < i_m=j\}$, set
    \[
    u_I :=u_{i_1 i_2}u_{i_2 i_3} \dots u_{i_{m-1} i_m},
    \]
    (where it is understood that if $I=\{i<j\}$, then $u_I =u_{ij}$).

  \item Define $v_{ij}^\sigma\colon c_{ij}^\sigma\to A$, using
    modified barycentric coordinates on $c_{ij}^\sigma$ (see
    \ref{not:mod-bary-coords}), as follows.  For $\mathbf{s}=(s_0,
    \dots, s_i = 1, \dots, s_j = 1, \dots, s_r)\in c_{ij}^\sigma$ let
    \begin{displaymath}
      v_{ij}^\sigma(\mathbf{s}) := \sum_I\, \lambda_I(\mathbf{s}) u_I,
    \end{displaymath}
    where
    \[
    \lambda_I(\mathbf{s})=\lambda_I^{\sigma}(\mathbf{s})=\prod_{i\leq
      k \leq j}s_k'
    \]
    with $s'_k=s_k$ if $k\in I$ and $s'_k=1-s_k$ if $k \notin I$.

    The sum above is over the subsets $I$ of $\{ i, \dots, j \}
    \subseteq \sigma^{(0)}$ that contain both $i$ and $j$ as above.
    One can identify the subsets $I$ with ascending paths from $i$ to
    $j$ that are contained in $\sigma$.  Let us note that $\sum_I\,
    \lambda_I(\mathbf{s})=1$ for $\mathbf{s}\in c_{ij}^\sigma$ since
    $\prod_{i<k<j}((s_k+(1-s_k))=1$.  We will see that the range of
    $v_{ij}^\sigma$ is actually contained in $\GL(A)$.  If $i > j$,
    let $v_{ij}^\sigma$ be the pointwise inverse of $v_{ji}^\sigma$.
  \item Corollary~\ref{cor:cocycle-well-defined} below shows that for
    each $\simplex{ i, j } \in \Lambda^{(1)}$ the collection of all
    $v_{ij}^\sigma$ above determines a function $v_{ij}\colon
    c_{ij}\to \GL(A)$.  We define
    \(
    \beta(\pi) := \{ v_{ij} \}.
    \)
  \end{enumerate}
\end{definition}

\begin{proposition}
  \label{prop:almost-flat-existence}
  There exist positive numbers $C_0''$ and $\delta_0$, depending only
  on $\Lambda$, such that if $0 < \delta < \delta_0$ and $\pi\colon
  \Gamma\to \GL(A)$ is an $(\mathcal{F}_\Lambda,
  \delta)$-representation of $\Gamma$, then $\beta(\pi)$ is a
  $C_0''\delta$-flat $\GL(A)$-coordinate bundle on $\Lambda$.
\end{proposition}

The rest of the section is devoted to the proof of
Proposition~\ref{prop:almost-flat-existence}.

\begin{remark}
  \label{rem:bundle-constr}
  The construction described in Definition~\ref{def:cocycle} is an
  attempt to have $v_{ij}$ be ``as constant as possible'' and equal to
  $u_{ij}$ at the barycenter of $\simplex{i, j}$
  (cf. \cite[Sec.~2]{Phillips-Stone86}).  The cocycle condition one
  might hope for would force relations of the form $u_{ij}u_{jk} =
  u_{ik}$, which do not necessarily hold since $\pi$ (and therefore
  $\bpi$) is only approximately multiplicative.  The definition of
  $v_{ij}^\sigma$ uses successive linear interpolation to account for
  this.  For example:
  \begin{enumerate}
  \item If $\sigma = \simplex{ 0, 1, 2 }$, then $v_{01}^\sigma =
    u_{01}$ and $v_{12}^\sigma =
    u_{12}$, but
    \[
    v_{02}^\sigma(s_0 = 1, s_1, s_2 = 1) =
    s_1u_{01}u_{12} + (1 - s_1)u_{02}.
    \]
  \item If $\sigma = \simplex{ 0, 1, 2, 3}$, then $v_{01}^\sigma=u_{01}$,
    $v_{12}^\sigma=u_{12}$,  $v_{23}^\sigma=u_{23},$
     \[
    v_{02}^\sigma(s_0 = 1, s_1, s_2 = 1) =
    s_1u_{01}u_{12} + (1 - s_1)u_{02}
    ,\]
     \[
    v_{13}^\sigma(s_1 = 1, s_2, s_3 = 1) =
    s_2u_{12}u_{23} + (1 - s_2)u_{13}
    ,\]
    and
    \begin{align*}
      v_{03}^\sigma(s_0=1, s_1, s_2, s_3=1) &=
      s_1s_2u_{01}u_{12}u_{13} + s_1(1 - s_2)u_{01}u_{13} +\mbox{}\\
      &\quad
      + (1 - s_1)s_2 u_{02}u_{23}
      + (1 - s_1)(1 - s_2)u_{03}.
    \end{align*}
  \end{enumerate}

  Notice that in the modified barycentric coordinates $(s_0, \dots,
  s_r)$ the barycenter of $\simplex{i ,j}$ is given by $s_i = s_j = 1$
  and $s_k = 0$ for all $k\not\in \{i, j\}$.  Therefore,
  $v_{ij}^\sigma(\simplex{i,j}\,\hat\,) = u_{ij}$ as desired.
\end{remark}

To start the construction, we first perturb $\pi$ slightly so that we
can deal with unitary elements instead of just invertible ones when
convenient.

\begin{proposition}
  \label{prop:perturb-qrep}
  Given an $(\mathcal{F}_\Lambda, \delta)$-representation $\pi\colon
  \Gamma\to \GL(A)$, $0<\delta < 1/7$, there exists a function
  $\bpi\colon \Gamma\to \U(A)$ such that
  \begin{enumerate}
  \item $\bpi(e) = 1_A$;
  \item $\bpi(\gamma) \in \U(A)$ for all $\gamma\in
    \mathcal{F}_\Lambda$;
  \item $\bpi(\gamma^{-1}) = \bpi(\gamma)^*$ for all $\gamma \in
    \mathcal{F}_\Lambda$;
  \item $\| \bpi(\gamma \gamma') - \bpi(\gamma)\bpi(\gamma') \| <
    70\delta$ for all $\gamma, \gamma' \in \mathcal{F}_\Lambda$ with
    $\gamma\gamma' \in \mathcal{F}_\Lambda$; and
  \item $\| \bpi(\gamma) - \pi(\gamma) \| < 20\delta$ for all $\gamma
    \in \mathcal{F}_\Lambda$.
  \end{enumerate}
\end{proposition}

The next lemma will be used in the proof.

\begin{lemma}
  \label{lem:polar-decomp-estimate}
  Let $\omega\colon \GL(A)\to \U(A)$ be given by $\omega(v) =
  v(v^*v)^{-1/2}$.  If $v\in \U(A)_\delta$, $0<\delta < 1/7$, then
  $\|\omega(v)-v\|<5\delta$.
\end{lemma}

\begin{proof}
  Observe that for $z\in A$
  \begin{equation}
    \label{eq:4}
    \| z - 1 \| \leq \theta < 1/2 \quad\Rightarrow\quad
    \| z^{-1} - 1 \| < 2\theta,
  \end{equation}
  using the Neumann series.

  Assume $v\in \U(A)_\delta$.  Then there is $u\in \U(A)$ such that
  $\|v - u\| < \delta$, so $\| vu^* - 1 \| < \delta$.  Let $w =
  vu^*$.  Then
  \[
    \omega(w) = vu^*(uv^*vu^*)^{-1/2}
    = vu^*u(v^*v)^{-1/2}u^*  = \omega(v) u^*
  \]
  and thus $\|\omega(w) - w\| = \| \omega(v) - v\|$.

  Now $\|w\| < 1 + \delta$, so $\| w^*w - 1 \| \leq \|(w^*-1) w\| +
  \|w - 1\| < \delta(1+\delta) + \delta < 15\delta/7$.  Therefore, by
  \eqref{eq:4}, $\| (w^*w)^{-1/2} - 1 \|\leq \| (w^*w)^{-1} - 1 \| <
  30\delta/7$.  Finally,
  \[
  \| \omega(v) - v\|
  = \| w(w^*w)^{-1/2} - w \| < (1+\delta)30\delta/7 < 5\delta.
  \qedhere
  \]
\end{proof}

\begin{proof}[Proof of Proposition~\ref{prop:perturb-qrep}]
  The idea is simple: define $\bpi = \omega \circ \sigma$ where
  $\omega$ is as in Lemma~\ref{lem:polar-decomp-estimate} and
  $\sigma\colon \Gamma\to A$ is the function
  \[
  \sigma(\gamma) =
  \frac{\pi(\gamma) + \pi(\gamma^{-1})^*}{2}.
  \]
  We check the required properties.

  Let $\gamma\in \mathcal{F}_\Lambda$.  First we prove that $\|
  \pi(\gamma^{-1}) - \pi(\gamma)^{-1} \| < 2\delta$ and $\|
  \pi(\gamma)^{-1} - \pi(\gamma)^* \| < 3\delta$.

  Because $\pi(\gamma) \in \U(A)_\delta$, we can write $\pi(\gamma) =
  uv$ for some $u\in \U(A)$ and $v\in \GL(A)$ with $\|v-1 \| <
  \delta$. Then $\| \pi(\gamma)^{-1}-u^* \|=\| v^{-1}-1 \|<2\delta$ by
  \eqref{eq:4} and hence $\| \pi(\gamma)^{-1} \| < 1 +2\delta$.
  Therefore,
  \begin{align*}
    \| \pi(\gamma^{-1}) - \pi(\gamma)^{-1} \|
    &=
    \| \left(\pi(\gamma^{-1})\pi(\gamma) - 1\right)\pi(\gamma)^{-1} \|\\
    &\leq
    \| \pi(\gamma^{-1})\pi(\gamma) - \pi(\gamma^{-1}\gamma) \|
    \| \pi(\gamma)^{-1} \|\\
    &\leq
    \delta(1 + 2\delta) < 2\delta.
  \end{align*}
  It is just as plain to see that
  \[
  \|
  \pi(\gamma)^{-1} - \pi(\gamma)^* \|
  =
  \|(uv)^{-1}-(uv)^*\|
  =
  \|v^{-1}-v^*\|
  \leq
  \|v^{-1}-1\|+\|v^*-1\|
  <3\delta,
  \]
  as claimed.  Using these bounds we see that
  \begin{align}
    \begin{split}
      \label{eq:5}
      \| \sigma(\gamma) - \pi(\gamma) \| &=
      \frac{1}{2} \| \pi(\gamma^{-1})^* - \pi(\gamma) \|
      =
      \frac{1}{2} \| \pi(\gamma^{-1}) - \pi(\gamma)^* \|\\
      &\leq \frac{1}{2} \| \pi(\gamma^{-1}) - \pi(\gamma)^{-1} \| +
      \frac{1}{2} \| \pi(\gamma)^{-1} - \pi(\gamma)^* \|\\
      &< 5\delta/2.
    \end{split}
  \end{align}
 Thus
 \begin{equation}
   \label{eq:6}
    \dist( \sigma(\gamma), \U(A) )
    < 5\delta/2 + \dist( \pi(\gamma), \U(A) )
    < 7\delta/2<1/2.
  \end{equation}

  In particular $\sigma(\gamma)\in \GL(A)$.  Items (1) and (2) in the
  statement of the proposition are immediate.  For (3) observe that if
  $z\in \GL(A)$, then $\omega(z^*) = \omega(z)^*$.  It follows that
  $\bpi(\gamma^{-1}) = \omega( \sigma(\gamma^{-1}) ) = \omega(
  \sigma(\gamma)^* ) = \omega( \sigma(\gamma) )^*=\bpi( \gamma )^*$.

  We deal with (5).  From \eqref{eq:6} and
  Lemma~\ref{lem:polar-decomp-estimate} we obtain $\| \bpi(\gamma) -
  \sigma(\gamma) \| < 35\delta/2$.  Together with \eqref{eq:5} this
  gives
  \[
  \| \bpi(\gamma) - \pi(\gamma) \|
  \leq
  \| \bpi(\gamma) - \sigma(\gamma) \| + \| \sigma(\gamma) -
  \pi(\gamma) \| < 20\delta.
  \]

  We are left with (4).  Suppose $\gamma,\gamma' \in
  \mathcal{F}_\Lambda$ are such that $\gamma\gamma' \in
  \mathcal{F}_\Lambda$.  Then
  \begin{align*}
    \| \bpi(\gamma)\bpi(\gamma') - \bpi(\gamma\gamma') \|
    &\leq
    \| \bpi(\gamma) - \pi(\gamma) \| \| \bpi(\gamma') \|
    + \| \pi(\gamma) \| \| \bpi(\gamma') - \pi(\gamma') \|
    +\mbox{}\\
    &\quad
    + \| \pi(\gamma)\pi(\gamma') - \pi(\gamma\gamma') \|
    + \| \bpi(\gamma\gamma') - \pi(\gamma\gamma') \|\\
    &<
    20\delta + (1+\delta)20\delta + \delta +20\delta\\
    &<
    70\delta.\qedhere
  \end{align*}
\end{proof}

The next proposition allows us to use induction on the number of
vertices of $\sigma$ in the proofs that follow.

\begin{proposition}
  \label{Prop_restriction}
  If $\sigma\subset\widetilde{\sigma}$ are simplices of $\Lambda$ and
  $i<j$ are vertices of $\sigma$, then the restriction of
  $v_{ij}^{\widetilde{\sigma}}$ to $c_{ij}^\sigma$ is equal to
  $v_{ij}^{\sigma}$.
\end{proposition}

\begin{proof}
  We may assume that $\widetilde{\sigma}=\sigma \cup \{{l}\}$ is a
  simplex of $\Lambda$ that has $\sigma$ as one of its faces and $l
  \notin \sigma$.  If $\mathbf{s}=(s_0, \dots, s_l, \dots, s_r) \in
  c_{ij}^{\widetilde{\sigma}}$, then $s_i=s_j=1$ and moreover
  $\mathbf{s}\in c_{ij}^{\sigma}$ precisely when $s_l=0$.  Let $I$ be
  a subset of $\{ i, \dots, j\}$ that contains both $i$ and $j$ as
  above.  If either $l<i $ or $j<l$, then $\sigma$ and
  $\widetilde{\sigma}$ have exactly the same set of increasing paths
  from $i$ to $j$ and hence
  $v_{ij}^{\widetilde{\sigma}}(\mathbf{s})=v_{ij}^{\sigma}(\mathbf{s})$
  for $\mathbf{s}\in c_{ij}^{\sigma}$ by Definition~\ref{def:cocycle}.
  (In fact, $v_{ij}^{\widetilde{\sigma}}(s_0, \dots, s_l, \dots, s_r)
  = v_{ij}^{\sigma}(s_0, \dots, s_l=0, \dots, s_r)$, again by
  Definition~5.3.)

  Suppose now that $i<l<j$. Let $I = \{ i=i_0<i_1 < i_2 < \cdots <
  i_m=j\}$ be an increasing path in $\widetilde{\sigma}$. If $l \notin
  I$ and $\mathbf{s}\in c_{ij}^\sigma=\{\mathbf{s}\in
  c_{ij}^{\widetilde{\sigma}}\colon s_l=0\},$ then
  $\lambda_I^{\widetilde{\sigma}}(\mathbf{s})u_I=
  (1-s_l)\lambda_I^{\sigma}(\mathbf{s})u_I=\lambda_I^{\sigma}(\mathbf{s})u_I.$
  On the other hand if $l\in I$, then
  $\lambda_I^{\widetilde{\sigma}}(\mathbf{s})=0$ since $s_l$ is one of
  its factors. The statement follows now immediately from by
  Definition~\ref{def:cocycle} since
  \[
  v_{ij}^{\widetilde{\sigma}}(\mathbf{s}) :=
  \sum_{ l\in I}\, \lambda^{\widetilde{\sigma}}_I(\mathbf{s})
  u_I+\sum_{ l\notin I}\,
  \lambda^{\widetilde{\sigma}}_I(\mathbf{s})u_I.\qedhere
  \]
\end{proof}

\begin{proposition}\label{Prop_cocycle}
  If $i<l<j$ are vertices of a simplex $\sigma$ of $\Lambda$, then
  $v^\sigma_{ij}(\mathbf{s})=v_{il}^{\sigma}(\mathbf{s})
  v_{lj}^{\sigma}(\mathbf{s})$ for all $\mathbf{s}\in
  c^\sigma_{ilj}=c^\sigma_{i}\cap c^\sigma_{l}\cap c^\sigma_{j}$.
\end{proposition}

\begin{proof}
  Let $I = \{ i=i_0<i_1 < i_2 < \cdots < i_m=j\}$ be an increasing
  path in ${\sigma}$.

  If $l \notin I$, then $\lambda_I^{\sigma}(\mathbf{s})=0$ since
  $1-s_l=0$ is one of its factors.  On the other hand, if $l\in I$,
  say
  \[
  I = \{ i=i_0<\dots<i_{k-1} < i_k=l<i_{k+1}<\cdots < i_m=j\},
  \]
  then letting
  \[
  I'=\{ i_0<\dots<i_{k-1} < i_k\}
  \quad\text{and}\quad
  I''=\{i_k<i_{k+1}<\dots < i_m\}
  \]
  we see that
  \[
  \lambda_I^{\sigma}( \mathbf{s} )u_I
  =
  \lambda_{I'}^{\sigma}(\mathbf{s})s_l
  \lambda_{I''}^{\sigma}(\mathbf{s})u_{I'}u_{I''}
  =
  \lambda_{I'}^{\sigma}(\mathbf{s})u_{I'}
  \cdot \lambda_{I''}^{\sigma}(\mathbf{s})u_{I''}.
  \]
  The statement now follows from Definition~\ref{def:cocycle} since
  \[
  v_{ij}^{\sigma}(\mathbf{s}) := \sum_{ l\in I}\,
  \lambda^{\sigma}_I(\mathbf{s}) u_I+\sum_{ l\notin I}\,
  \lambda^{\sigma}_I(\mathbf{s})u_I.\qedhere
  \]
\end{proof}

\begin{corollary}
  \label{cor:cocycle-well-defined}
  The family of functions $\{ v_{ij}^\sigma \mid i,j\in \sigma\}$
  yields a continous function $v_{ij}\colon c_{ij} = \bigcup_\sigma
  c_{ij}^\sigma\to \GL(A)$ such that
  $v_{il}(\mathbf{s})v_{lj}(\mathbf{s})=v_{ij}(\mathbf{s})$ for all
  $\mathbf{s}\in c_{ilj}= c_i\cap c_l\cap c_j$.
\end{corollary}

\begin{proof}
  Proposition~\ref{Prop_restriction} shows that if two simplices
  $\sigma$ and $\sigma'$ contain $\{i,j\}$, then
  $v^{\sigma}_{ij}=v^{\sigma'}_{ij}=v^{\sigma\cap \sigma'}_{ij}$ on
  $c_{ij}^{\sigma \cap \sigma'}$, so that $v_{ij}$ is
  well-defined. The cocycle condition follows from Proposition
  ~\ref{Prop_cocycle}. It remains to show that $v_{ij}$ takes values
  in $\GL(A)$.  This will follow from the estimate
  \begin{equation}
    \label{eq:7}
    \| v_{ij}^\sigma(\mathbf{s}) - u_{ij} \|
    <
    70 L \delta < 1/2
  \end{equation}
  that we now verify.

  If $I = \{ i=i_1 < i_2 < i_3 < \dots < i_m=j\}$ and $ u_I :=u_{i_1
    i_2}u_{i_2 i_3} \dots u_{i_{m-1} i_m} $ are as in
  Definition~\ref{def:cocycle}, then show that $\| u_I - u_{ij} \| <
  70 m\delta$ by induction on $m$.  This is trivial if $m=2$, since in
  that case $u_I = u_{ij}$.  For the inductive step, we use the
  estimate $\|u_{i_k i_{k+1}}u_{i_{k+1} i_{k+2}}-u_{i_{k}
    i_{k+2}}\|<70\delta$ proved in
  Proposition~\ref{prop:perturb-qrep}(4).  Because
  $v_{ij}^\sigma(\mathbf{s}) := \sum_I\, \lambda^\sigma_I(\mathbf{s})
  u_I$, the estimate \eqref{eq:7} follows since $\sum_I\,
  \lambda^\sigma_I(\mathbf{s})=1$ for $\mathbf{s}\in c_{ij}^\sigma$.
\end{proof}

\begin{proof}[Proof of Proposition~\ref{prop:almost-flat-existence}]
  Let $\delta_0 = 1/140L$.

  Corollary~\ref{cor:cocycle-well-defined} all but implies
  Proposition~\ref{prop:almost-flat-existence}.  To complete the
  proof, we need to verify the almost flatness condition.  Assume $i <
  j$ are vertices as in the proof of
  Corollary~\ref{cor:cocycle-well-defined}.  We have seen that $\|
  v_{ij}(x) - u_{ij} \| < 70 L \delta$ for all $x \in c_{ij}$.  Since
  $u_{ij}$ is a unitary and since $0 < \delta < 1 / 140L$ by
  hypothesis, we can apply \eqref{eq:4} to see that
  \[
  \| v_{ji}(x) - u_{ji} \|
  =
  \| ( v_{ij}(x) )^{-1} - u_{ij}^{-1} \|
  \leq
  2\| v_{ij}(x) - u_{ij} \|
  < 140 L \delta.
  \]
  We conclude that $\beta(\pi) = \{ v_{ij} \}$ is $C_0''\delta$-flat
  where $C_0'' = 280 L \delta$, completing the proof of
  Proposition~\ref{prop:almost-flat-existence}.
\end{proof}

\section{Proofs of Theorems~\ref{thm:main-correspondence}
  and~\ref{thm:approx-inverse}}
\label{sec:proofs-main-results}

Most of the work needed to prove
Theorems~\ref{thm:main-correspondence} and~\ref{thm:approx-inverse}
was done in Sections~\ref{sec:bundle-to-qrep}
and~\ref{sec:qrep-to-bundle}.  What is left is basically bookkeeping
related the various constants defined so far, but it is somewhat
technical due to the nature of the definitions of $\alpha$ and
$\beta$.

\begin{proof}[Proof of Theorem~\ref{thm:main-correspondence}]
  The definitions of $\alpha$ and $\beta$ are given in
  Sections~\ref{sec:bundle-to-qrep} and \ref{sec:qrep-to-bundle}.
  Propositions~\ref{prop:qrep-existence} and
  \ref{prop:almost-flat-existence} show the existence of $\delta_0,
  \varepsilon_0 > 0$ and a constant $\max\{ C_0', C_0'' \}$ satisfying
  parts (1) and (2) of the theorem.  We will actually set $C_0 =
  \max\{ C_0', 2C_0'' + 40, 4^{L + 1}(K+1)\}$ where $K$ is provided by
  Lemma~\ref{lem:bound-on-R}.

  We prove (3).  Let $\pi = \alpha(\mathbf{v})$, $\pi' =
  \alpha(\mathbf{v'})$.  Let $\simplex{i, j}\in \Lambda^{(1)}$ be such
  that
  \[
  \| \pi(\gamma_{ij}) - \pi'(\gamma_{ij}) \|
  =
  \max_{\gamma \in \mathcal{F}_\Lambda}
  \| \pi(\gamma) - \pi'(\gamma) \|
  =
  d(\pi, \pi').
  \]
  As in Definition~\ref{def:quasi-rep-constr}, let $I = (i_0, \dots,
  i)$ be the unique path along $T$ from $i_0$ to $i$ and $J = (i_0,
  \dots, j)$ be the unique path from $i_0$ to $j$.  Then
  \begin{align*}
    d(\pi, \pi')
    &\leq
    \| \tilde{\pi}(s(\gamma_{ij})) - \tilde{\pi}(\simplex{i, j}) \|
    +
    \| \tilde{\pi}'(\simplex{i, j}) - \tilde{\pi}'(s(\gamma_{ij})) \|
    +\\
    &\quad\
    + \| \tilde{\pi}( \simplex{i, j} ) - \tilde{\pi}'( \simplex{i, j} ) \|
    \\
    &\leq
    \| \tilde{\pi}\big( s(\gamma_{ij}) \simplex{i, j}^{-1} \big)
    - 1\| \cdot \| \tilde{\pi}(\simplex{ i, j }) \|
    +\mbox{}\\
    &\quad\
    +
    \| \tilde{\pi}'\big( s(\gamma_{ij}) \simplex{ i, j }^{-1} \big)-1 \|
    \cdot \| \tilde{\pi}(\simplex{ i, j }) \|
    +
    \| \baryv{I} \baryv{ij} \baryv{J}^{-1} -
    \baryv{I}' \baryv{ij}' (\baryv{J}')^{-1} \|.
  \end{align*}
  From Lemma~\ref{lem:bound-on-R} we get that $\| \tilde{\pi}\big(
  s(\gamma_{ij}) \simplex{ i, j }^{-1} \big)-1 \| \leq K\varepsilon$,
  where $K > 0$ depends only on $\Lambda$, $T$, $i_0$, and $s$.  The
  same bound holds with $\tilde{\pi}'$ instead of $\tilde{\pi}$.
  Using this and the estimates $\|\baryv{kl}\|<1+\varepsilon$,
  $\|\baryv{kl}-\baryv{kl}'\|<d(\mathbf{v}, \mathbf{v'})$ for
  $\simplex{ k, l } \in \Lambda^{(1)}$, we see that
  \begin{displaymath}
   d(\pi, \pi') \leq 2\cdot K\varepsilon\cdot (1 + \varepsilon)^{2L+1}
    + (1 + \varepsilon)^{2L}d(\mathbf{v}, \mathbf{v'}).
  \end{displaymath}
  Since $(1 + \varepsilon)^{2L}<1+2^{2L}\varepsilon$ and
  $d(\mathbf{v}, \mathbf{v'})<2+2\varepsilon<4$ we have
  \[
  2 K\varepsilon(1 + \varepsilon)^{2L+1}+(1 +
  \varepsilon)^{2L}d(\mathbf{v},
  \mathbf{v'})<2^{2L+2}K\varepsilon+2^{2L+2}\varepsilon+d(\mathbf{v},
  \mathbf{v'}).
  \]
  Thus $d(\pi, \pi')< C_0\varepsilon+d(\mathbf{v}, \mathbf{v'})$.

  For part (4), recall that $d(\pi, \pi') = \max_{\gamma \in
    \mathcal{F}_\Lambda} \| \pi(\gamma) - \pi(\gamma) \|$.  Let
  $\mathbf{v} =\beta(\pi)$ and $\mathbf{v'} = \beta(\pi')$ (these are
  $C_0''\delta$-flat $\GL(A)$-coordinate bundles by
  Proposition~\ref{prop:almost-flat-existence}).  Recall that their
  definition (see Definition~\ref{def:cocycle}) makes use of the maps
  $\bpi$ and $\bpi'$ (given by Proposition~\ref{prop:perturb-qrep})
  respectively, and that $u_{ij} = \bpi( \gamma_{ij} ) = \baryv{ij}$
  etc.  For $\simplex{i, j} \in \Lambda^{(1)}$ and $x \in c_{ij}$ we
  estimate
  \begin{align*}
   \| v_{ij}(x) - v_{ij}'(x) \| &\leq
    \| v_{ij}(x) - u_{ij} \|
    + \| u_{ij}' - v_{ij}'(x) \|
    + \| u_{ij} - u_{ij}'\|\\
    &<
    C_0''\delta + C_0''\delta
    + \| u_{ij} - \pi(\gamma_{ij}) \|
    + \| \pi'(\gamma_{ij}) - u_{ij}' \| +\mbox{}\\
    &\quad
    + \| \pi(\gamma_{ij}) - \pi'(\gamma_{ij}) \|\\
    &<
    2C_0''\delta
    + 20\delta + 20\delta + d(\pi, \pi')\\
    &< C_0\delta + d(\pi, \pi').
  \end{align*}
  It follows that $d(\mathbf{v}, \mathbf{v'}) <C_0\delta + d(\pi,
  \pi')$. \qedhere
\end{proof}

\begin{proof}[Proof of Theorem~\ref{thm:approx-inverse}]
  We prove (1) from the statement of the theorem first.  Let
  $\varepsilon_1 = 1/ 140 L C_0$, $\delta_1 = 1/ 140 L C_0$, and $C_1
  = 70 K C_0^2$.  ($C_0$ is provided by
  Theorem~\ref{thm:main-correspondence}, $K$ by
  Lemma~\ref{lem:bound-on-R} and $L$ by Notation~\ref{not:ell-and-L}).
  Let $0 < \varepsilon < \varepsilon_1$ and suppose $\mathbf{v} = \{
  v_{ij}\colon c_{ij}\to \GL(A) \}$ is an $\varepsilon$-flat
  $\GL(A)$-coordinate bundle on $\Lambda$.  Let $\pi =
  \alpha(\mathbf{v})$ and $\mathbf{v'} = \{ v'_{ij}\colon c_{ij}\to
  \GL(A) \} = \beta(\pi)$.  Observe that $\pi$ is an
  $(\mathcal{F}_\Lambda, C_0\varepsilon)$-representation and
  $C_0\varepsilon < 1/140L$ so that the construction of $\beta(\pi)$
  from Section~\ref{sec:qrep-to-bundle} may be used.  We want to prove
  that
  \begin{displaymath}
    d( \mathbf{v}, \mathbf{v}' ) =
    \max_{ \simplex{i, j}\in \Lambda^{(1)} }
    \max_{ x\in c_{ij} }
    \| v_{ij}(x) - v'_{ij}(x) \| < C_1\varepsilon.
  \end{displaymath}

  Recall the notation $\baryv{ij} = v_{ij}( \bary{i, j} )$ from
  \ref{not:baryv}.  Since $\mathbf{v}$ is $\varepsilon$-flat and
  $\mathbf{v'}$ is $C_0^2\varepsilon$-flat, it follows that
  \begin{equation}
    \label{eq:8}
    d( \mathbf{v}, \mathbf{v'} )
    < \max_{ \simplex{i,j} \in \Lambda^{(1)} } (
    \varepsilon + \| \baryv{ij} - \baryv{ij}' \| + C_0^2\varepsilon ).
  \end{equation}

  Let $\simplex{i, j} \in \Lambda^{(1)}$ and set $g_{ij} :=
  s(\gamma_{ij}) \cdot \simplex{i, j}^{-1}$ as in
  Lemma~\ref{lem:bound-on-R}.  Applying the definition of
  $\tilde{\pi}$ (see Equation~\eqref{eq:1}) and the fact that
  $\mathbf{v}$ is normalized
  (Definition~\ref{def:normalized-coord-bundle}), we obtain
  \[
  \pi(\gamma_{i j})
  = \tilde{\pi}( s( \gamma_{i, j} ) )
  = \tilde{\pi}( \simplex{i, j} ) \tilde{\pi}(g_{ij})
  = \baryv{i j} \tilde{\pi}(g_{ij})
  \]
  The definition of $\beta(\pi)$ shows that $\baryv{ij}' = \bpi(
  \gamma_{ij} )$ and Proposition~\ref{prop:perturb-qrep} implies $\|
  \pi( \gamma_{i j} ) - \bpi( \gamma_{i j} ) \| < 20 C_0 \varepsilon$.
  Thus
  \begin{align}
    \begin{split}
      \label{eq:9}
      \| \baryv{ij} - \baryv{ij}' \| &< \| \baryv{ij} - \pi(
      \gamma_{ij} ) \|
      + \| \pi( \gamma_{i j} ) - \baryv{ij}' \|\\
      &= \| \baryv{ij}( 1 - \tilde{\pi}(g_{ij})) \|
      + \| \pi( \gamma_{ij} ) - \bpi( \gamma_{ij} ) \|\\
      &< \| \baryv{ij} \| \| 1 -
      \tilde{\pi}(g_{ij}) \| + 20C_0\varepsilon\\
      &< (1 + \varepsilon) \| 1 - \tilde{\pi}(g_{ij}) \| +
      20 C_0 \varepsilon.
    \end{split}
  \end{align}
  Lemma~\ref{lem:bound-on-R} guarantees that $\| 1_A -
  \tilde{\pi}(g_{ij}) \| < K\varepsilon$.  In combination with
  \eqref{eq:8} and \eqref{eq:9} this proves that
  \[
  d( \mathbf{v}, \mathbf{v}' ) <
  \varepsilon + C_0^2\varepsilon +
  (1 + \varepsilon)K \varepsilon +
  20 C_0 \varepsilon < C_1 \varepsilon.
  \]

  We prove (2) from the statement of the theorem.  Let $0 < \delta <
  \delta_1$ and suppose $\pi\colon \Gamma\to \GL(A)$ is an
  $(\mathcal{F}_\Lambda, \delta)$-representation.  Let $\mathbf{v} =
  \{ v_{ij} \} = \beta(\pi)$ (this is a $C_0\delta$-flat
  $\GL(A)$-coordinate bundle) and let $\bpi\colon \Gamma\to \U(A)$ be
  given by Proposition~\ref{prop:perturb-qrep}.  Let also $\pi' =
  \alpha(\mathbf{v})$ (this is an $(\mathcal{F}_\Lambda,
  C_0^2\delta)$-representation of $\Gamma$ to $\GL(A)$).  We want to
  prove that
  \[
  d( \pi, \pi' ) = \max_{\gamma\in \mathcal{F}_\Lambda}
  \| \pi(\gamma) - \pi'(\gamma) \| < C_1\delta.
  \]

  Suppose $\simplex{i, j} \in \Lambda^{(1)}$.  As above, we may write
  $s(\gamma_{i j})=\simplex{i, j} \cdot g_{ij}$ where by
  Lemma~\ref{lem:bound-on-R}, $\| \tilde{\pi}'(g_{ij}) - 1_A \| < K
  C_0\delta$.

  First notice that $\baryv{ij} = \bpi(\gamma_{ij})$ by the definition
  of $\beta(\pi) = \mathbf{v}$ (Definition~\ref{def:cocycle}).  Let
  $I$ be the unique path along $T$ from $i_0$ to $i$ and $J$ be the
  unique path from $i_0$ to $j$.  Observe that $\mathbf{v}$ is
  normalized since $\bpi( e ) = 1_A$; hence $\baryv{I} = 1_A =
  \baryv{J}$ because $I$ and $J$ are paths in the tree $T$.  Then, by
  the definition of $\alpha(\mathbf{v}) = \pi'$,
  \begin{displaymath}
    \pi'( \gamma_{i j} )
    = \tilde{\pi}' \big( s(\gamma_{i, j}) \big)
    = \tilde{\pi}'( \simplex{i, j} ) \cdot \tilde{\pi}'(g)
    = \baryv{I} \baryv{ij} \baryv{J}^{-1} \cdot \tilde{\pi}'(g)
    = \baryv{ij} \cdot \tilde{\pi}'(g)
    = \bpi(  \gamma_{i j}  ) \cdot \tilde{\pi}'(g).
  \end{displaymath}
  Therefore
  \begin{align*}
    \| \pi'( \gamma_{ij} ) - \pi( \gamma_{ij} ) \|
    &\leq \| \pi'(\gamma_{ij}) - \pi(\gamma_{ij}) \tilde{\pi}'(g) \|
    + \| \pi(\gamma_{ij}) \tilde{\pi}'(g) - \pi(\gamma_{ij}) \|\\
    &\leq \| \bpi(\gamma_{ij}) - \pi(\gamma_{ij})\|
    \| \tilde{\pi}'(g) \| + \| \pi(\gamma_{ij}) \|
    \| \tilde{\pi}'(g) - 1_A \|\\
    &< 20\delta(1 + K C_0^2\delta)
    + (1 + \delta)(K C_0^2\delta)\\
    &< C_1 \delta. \qedhere
  \end{align*}
\end{proof}

\section{Almost flat bundles}\label{sec:7}
The goal of this section is to connect the notion of almost flat
coordinate bundle from Definition~\ref{def:af-bundle}, which is
defined using simplicial structure and involves cocycles defined on
closed sets, with the notion of almost flat bundle over a compact
space from Definition~\ref{definition:to_alm_flat} below.

Almost flat bundles and $K$-theory classes appeared in the work Gromov
and Lawson \cite{Gromov-Lawson83}, of Connes, Gromov, and Moscovici
\cite{Connes-Gromov-etal90, Moscovici91, Skandalis91}.  In these
references a vector bundle over a Riemannian manifold is called
$\varepsilon$-flat if there is a metric-preserving connection with
curvature of norm less than $\varepsilon$.  Almost flat $K$-theory
classes have been studied in different contexts in
\cite{Manuilov-Mishchenko01, Mishchenko-Teleman05, Hanke-Schick06,
  Dadarlat12, Dadarlat14, Carrion-Dadarlat13}.  We adapt the
definition to bundles over topological spaces as in \cite{Dadarlat14}
and connect this with the version for simplicial complexes by proving
Proposition ~\ref{prop:bundle_extensions}.

Let $X$ be a compact space and let $\mathcal{V}=\{V_i\}$ be a finite
open cover of $X$.  A {\v C}ech 1-cocycle $\{ v_{ij}\colon V_i\cap
V_j\to \GL(A) \}$ satisfies $v_{ij}(x) =v_{ji}(x)^{-1}$ for all $x\in
V_i\cap V_j $ and $v_{ik}(x) = v_{ij}(x) v_{jk}(x)$ for all $x\in
V_i\cap V_j \cap V_k$.

\begin{definition}
  \label{definition:to_alm_flat}
  Let $\varepsilon\geq 0$.
  \begin{enumerate}
  \item A {\v C}ech 1-cocycle $\{ v_{ij}\colon V_i\cap V_j\to \GL(A)
    \}$ is \emph{$\varepsilon$-flat}, if
    \begin{enumerate}
    \item $v_{ij}(x) \in \U(A)_\varepsilon$ for all $x\in V_i\cap V_j
      $ ; and
    \item $\| v_{ij}(x) - v_{ij}(y) \| < \varepsilon$ for all $x, y
      \in V_i\cap V_j $.
    \end{enumerate}
  \item A principal $\GL(A)$-bundle $E$ over $X$ is
    $(\mathcal{V},\varepsilon)$-\emph{flat} if its isomorphism class
    is represented by an $\varepsilon$-flat cocycle $\{ v_{ij}\colon
    V_i\cap V_j\to \GL(A) \}$.
  \end{enumerate}
\end{definition}
It is clear that if $\mathcal{V}'$ is an open cover that refines
$\mathcal{V}$, then the restriction of $\{ v_{ij}\}$ to $\mathcal{V}'$
is also $\varepsilon$-flat.

We now establish a result that connects the notion of
$\varepsilon$-flat {\v C}ech 1-cocycles from
Definition~\ref{definition:to_alm_flat} with the notion of
$\varepsilon$-flat coordinate bundle in the simplicial sense as given
in Definition~\ref{def:af-bundle}.  Suppose that $X=|\Lambda|$ is the
geometric realization of a finite simplicial complex $\Lambda$.
Recall that $X$ has a (closed) cover $\mathcal{C}_\Lambda$ given by
dual cells $c_i$; see Section~\ref{sec:basic-def-not}.  Let $d$ be the
canonical metric for the topology of $X$ obtained using barycentric
coordinates.  Fix a sufficiently small number $\nu>0$ such that if we
set $V_i=\{x\in X \colon \dist(x,c_i)<\nu\},$ then for any finite
intersection
\begin{equation}
  \label{eq:intersection-nu0}
  V_{i_1}\cap V_{i_2}\cap \cdots \cap V_{i_k}
  \neq \emptyset
  \quad \Leftrightarrow\quad
  c_{i_1}\cap c_{i_2}\cap \cdots \cap c_{i_k}
  \neq \emptyset.
\end{equation}

Note that if $\{v_{ij}\}$ is as in
Definition~\ref{definition:to_alm_flat} and the cover $\{V_i\}$
satisfies \eqref{eq:intersection-nu0}, then the restriction of
$\{v_{ij}\}$ to $c_i\cap c_j \subset V_i\cap V_j$ is an
$\varepsilon$-flat coordinate bundle.
Proposition~\ref{prop:correspondence_af} below allows us reverse this
operation.

\begin{proposition}
  \label{prop:correspondence_af}
  There are numbers $\ep_0>0$ and $r>0$, depending only on $\Lambda$,
  such that for any $0<\ep<\ep_0$, any $\varepsilon$-flat
  $\GL(A)$-coordinate bundle $\{ v_{ij}\colon c_{i}\cap c_{j}\to
  U(A)_{\ep}\}$ on $\Lambda$, and any $\nu>0$ satisfying
  \eqref{eq:intersection-nu0}, there is an $r\ep$-flat cocycle
  $\{\tv_{ij}\colon V_i\cap V_j\to U(A)_{r\ep}\}$ that extends
  $v_{ij}.$
\end{proposition}

Proposition~\ref{prop:correspondence_af} is a direct consequence of
Proposition~\ref{prop:bundle_extensions} below, whose content is of
independent interest in connection with extension properties of
principal bundles.

Let $Y$ be a closed subspace of a compact metric space $X$ and let
$\{U_i\}_{i=1}^n$ be a closed cover of $Y$.  For $\nu>0$ and $i\in\{1,
\dots, n\}$ let
\[
\tu_i:=\{x\in X : \dist(x, U_i)\leq \nu\}
\]
and set $\widetilde{Y}=\bigcup_{i=1}^n \tu_i$.  Fix $\nu>0$ small
enough such that for any finite intersection
\begin{equation}
  \label{eq:intersection-nu}
  \tu_{i_1}\cap \tu_{i_2}\cap \cdots \cap \tu_{i_k}
  \neq \emptyset
  \quad \Leftrightarrow\quad
  U_{i_1}\cap U_{i_2}\cap \cdots \cap U_{i_k}
  \neq \emptyset
\end{equation}

\begin{proposition} \label{prop:bundle_extensions}
  \mbox{}

  \begin{enumerate}
  \item For any cocycle $v_{ij}\colon U_i\cap U_j\to GL(A)$ on $Y$
    there exist $\nu>0$ satisfying \eqref{eq:intersection-nu} and a
    cocycle $\tv_{ij}\colon \tu_i\cap \tu_j\to \GL(A)$ that extends
    $v_{ij}$, i.e. $\tv_{ij}=v_{ij}$ on $U_i\cap U_j$.

  \item There exist $\ep_0\in (0,1)$ and a universal constant
    $r=r_n$
    that depends only on $n$ such that for any $0<\ep<\ep_0$, any
    $\ep$-flat cocycle $v_{ij}:U_i\cap U_j\to U(A)_{\ep}$ on $Y$, and
    any $\nu>0$ satisfying \eqref{eq:intersection-nu}, there is an
    $r\ep$-flat cocycle $\tv_{ij}\colon \tu_i\cap \tu_j\to
    U(A)_{r\ep}$ on $\widetilde{Y}$ which extends $v_{ij}.$
  \end{enumerate}
\end{proposition}

\begin{proof}
  We begin with the proof of (1) and will explain subsequently how to
  adapt the argument to prove (2) as well.  We prove (1) by induction
  on the cardinality $n$ of the cover. Suppose that the statement is
  true for any integer $\leq n-1$. Let $v_{ij}\colon U_i\cap U_j\to
  \GL(A)$ be given with $1\leq i,j \leq n$.  Set
  $Y_{n-1}:=\bigcup_{i=1}^{n-1} U_i$. By the inductive hypothesis,
  there exist $\nu>0$ satisfying \eqref{eq:intersection-nu} and a
  cocycle $\tv_{ij}\colon \tu_i\cap \tu_j\to \GL(A)$, $1\leq i,j \leq
  n-1$, which extends $v_{ij}$.  Thus the following condition,
  labelled as $(n-1)$, is satisfied:
  \begin{equation}
    \label{n-1}
    \tv_{rs}=\tv_{rt}\tv_{ts}
    \text{ \ on \ } \tu_{r}\cap\tu_{t}\cap\tu_{s}
    \text{ \ for all \ } 1\leq r\leq t\leq s \leq n-1.
    \tag{$n-1$}
  \end{equation}

  To pass from $n-1$ to $n$ we proceed again by induction on
  increasing $k\in L_n$, where $L_n$ is the set of those integers
  $1\leq k \leq n$ with the property that $U_k\cap U_n \neq
  \emptyset$.  The inductive hypothesis that we make is that the
  functions $\{v_{in}\colon i \leq k, i\in L_n\}$ extend to functions
  $\tv_{in}\colon \tu_i\cap \tu_n \to GL(A)$ such that the following
  conditions (depending on $k$) are satisfied:
  \begin{gather}
    \label{1-k}
    \tv_{in}=\tv_{ij}\tv_{jn}
    \text{ \ on \ }
    \tu_{i}\cap\tu_{j}\cap\tu_{n}
    \text{ \ for \ }
    i,j \leq k
    \text{ \ with \ }
    i,j \in L_n.
    \tag{$1,k$}\\
    \label{2-k}
    \tv_{in}=\tv_{ij}v_{jn}
    \text{ \ on \ }
    \tu_{i}\cap U_{j}\cap U_{n}
    \text{ \ for \ }
    i \leq k \leq j
    \text{ \ with \ }
    i,j \in L_n.
    \tag{$2,k$}
  \end{gather}

  If $L_n$ reduces to $\{n\}$, then we simply define $\tv_{nn}=1$ and
  we are done.  Assume $L_n$ contains more than one element.

  Let $\ell$ be the smallest element of $L_n$ (so $\ell<n$).  To
  construct $\tv_{\ell n}$ we first define an extension $v'_{\ell n}$
  of $v_{\ell n}$ on suitable closed subsets of $\tu_\ell\cap \tu_n$
  as follows:
  \begin{equation}
    \label{0-prime}
    v'_{\ell\,n}=\tv_{\ell\,j}v_{jn}
    \text{ \ on \ }
    \tu_{\ell}\cap U_{j}\cap U_{n}
    \text{ \ for all \ }
    \ell\leq j<n, \ j\in L_n.
    \tag{$0'$}
  \end{equation}
  Let us observe that $v'_{\ell\,n}$ is well-defined since if
  $\ell\leq i\leq j<n$, $i,j\in L_n,$ then
  $\tv_{\ell\,i}v_{in}=\tv_{\ell\,j}v_{jn}$ on $\tu_{\ell}\cap
  U_{i}\cap U_{j}\cap U_{n}$ if and only if
  $v_{in}=\tv_{i,\ell}\tv_{\ell j}v_{jn}$ on the same set.  In view of
  condition $(n-1)$ this reduces to the equality
  $v_{in}=\tv_{ij}v_{jn}$ on $U_{i}\cap U_{j}\cap U_{n}$, which holds
  true since $\tv_{ij}$ extends $v_{ij}$. By Tietze's theorem we can
  now extend the function $v'_{\ell\,n}$ defined by ($0'$) to a
  continuous function $\tv_{\ell\,n}\colon \tu_{\ell}\cap \tu_{n}\to
  A$. Since $\GL(A)$ is open in $A$ we will have that
  $\tv_{\ell\,n}(x)\in \GL(A)$ for all $x\in \tu_{\ell}\cap \tu_{n}$
  provided that $\nu$ is sufficiently small.  We need to very that
  $\tv_{\ell n}$ satisfies $(1,\ell)$ and $(2,\ell)$.  Condition
  $(1,\ell)$ amounts to $\tv_{\ell n}=\tv_{\ell\ell}\tv_{\ell n}$ on
  $\tu_{\ell}\cap\tu_{n}$ which holds since
  $\tv_{\ell\ell}=1$. Condition $(2,\ell)$ reduces to $\tv_{\ell
    n}=\tv_{\ell j}v_{jn}$ on $\tu_{\ell}\cap U_{j}\cap U_{n}$ for $
  \ell \leq j $ with $j \in L_n$. This holds true in view of ($0'$)
  and so the base case for the induction is complete.

  Fix $k\in L_n$, $k<n$ and suppose now that we have constructed
  $\tv_{in}\colon \tu_i\cap \tu_n\to \GL(A)$ for all $i\in L_n$ with
  $i\leq k$ such that the conditions ($1,k$) and ($2,k$) are
  satisfied.  Let $\ell \in L_n$ be the successor of $k$ in $L_n$. We
  may assume that $\ell<n$ for otherwise there is nothing to prove.
  We construct a map $\tv_{\ell n}$ on $\tu_{\ell}\cap\tu_{n}$ that
  satisfies the corresponding conditions $(1,\ell)$ and $(2,\ell)$ as
  follows.  The first step is to define an extension $v'_{\ell\,n}$ of
  $v_{\ell\,n}$ on suitable closed subsets of $\tu_{\ell}\cap \tu_n$
  as follows:
  \begin{gather}
    \label{1-prime}
    \tag{$1'$}
    v'_{\ell\,n}=\tv_{\ell\,i}\tv_{in}
    \text{ \ on \ }
    \tu_{\ell}\cap \tu_{i}\cap \tu_{n}
    \text{ \ for all \ }
    i \leq k, \ i \in L_n.\\
    \label{2-prime}
    \tag{$2'$}
    v'_{\ell\,n}=\tv_{\ell\,j}v_{jn}
    \text{ \ on \ }
    \tu_{\ell}\cap U_{j}\cap U_{n}
    \text{ \ for all \ }
    \ell\leq j<n, \ j\in L_n.
  \end{gather} %
  We need to observe that the conditions $(1')$ and $(2')$ are
  compatible so that $v'_{\ell,n}$ is well-defined and continuous.
  There are three cases to verify.  First we check that
  $\tv_{\ell,i}\tv_{in}=\tv_{\ell j}\tv_{jn}$ on $\tu_{\ell}\cap
  \tu_{i}\cap \tu_{j}\cap \tu_{n}$ for $i,j \leq k$, $i ,j\in L_n$.
  This is a consequence of conditions $(n-1)$ and $(1,k)$.  Second, we
  verify that $\tv_{\ell\,i}v_{in}=\tv_{\ell\,j}v_{jn}$ on
  $\tu_{\ell}\cap U_{i}\cap U_{j}\cap U_{n}$ for $\ell \leq i,j<n$,
  $i,j\in L_n$. Note that this holds if and only if
  $v_{in}=\tv_{i,\ell}\tv_{\ell j}v_{jn}$ on the same set.  In view of
  condition $(n-1)$ this reduces to the equality
  $v_{in}=\tv_{ij}v_{jn}$ on $U_{i}\cap U_{j}\cap U_{n}$, which holds
  true since $\tv_{ij}$ extends $v_{ij}$.  Finally we need to verify
  that $\tv_{\ell i}\tv_{in}=\tv_{\ell j}v_{jn}$ on $(\tu_{\ell}\cap
  \tu_{i}\cap \tu_{n})\cap (\tu_{\ell}\cap U_{j}\cap
  U_{n})=\tu_{\ell}\cap \tu_{i}\cap U_{j}\cap U_{n}$ for $i\leq k <
  \ell \leq j<n$, $i,j \in L_n$.  By $(n-1)$, this equality holds if
  and only if $\tv_{in}=\tv_{i\,j}v_{jn}$ on $\tu_{i}\cap U_{j}\cap
  U_{n}$.  The latter equality holds due to condition $(2,k)$ which is
  satisfied by the inductive hypothesis.  By Tietze's theorem we can
  now extend the function $v'_{\ell\,n}$ defined by ($1'$) and ($2'$)
  to a continuous function $\tv_{\ell\,n}\colon \tu_{\ell}\cap
  \tu_{n}\to A$. Since $\GL(A)$ is open in $A$ we will have that
  $\tv_{\ell\,n}(x)\in \GL(A)$ for all $\tu_{\ell}\cap \tu_{n}$
  provided that $\nu$ is sufficiently small.  It is clear that the
  functions $(\tv_{i n})_{i\leq \ell}$ satisfy the conditions
  $(1,\ell), (2,\ell)$ as a consequence of ($1'$), ($1,k$), ($2'$) and
  ($2,k$). This completes the inductive step from $k$ to $\ell$ and
  hence from $n-1$ to $n$. During this step we had to pass to a
  possibly smaller $\nu$ but this does not affect the conclusion.

  (2). The proof follows the pattern of the proof of (1) with one
  important modification.  Namely we use the following strengthened
  version of Tietze's theorem due to Dugunji \cite{Dugundji51}.  Let
  $X$ be an arbitrary metric space, $Y$ a closed subset of $X$, $A$ a
  locally convex linear space and $f\colon Y\to A$ a continuous
  map. Then there exists an extension $\widetilde{f}\colon X \to A$ of
  $f$ such that $\widetilde{f}(X)$ is contained in the convex hull of
  $f(Y)$.

  Fix a point $x_{ij}$ in each nonempty intersection $U_i\cap U_j$ and
  set $\baryv{ij}:=v_{ij}(x_{ij})$. Since the cocycle is $\ep$-flat,
  we have that $\|v_{ij}(x)-\baryv{ij}\|<\ep$.

  Let us define positive numbers $r(i,j)$ for $1 \leq i \leq j \leq n$
  as follows. If $i=j$, then $r(i,j)=1$. If $i<j$, $r(i,j)$ is defined
  by the following recurrence formula. Set $r_k=\max \{ r(i,j)\colon 1
  \leq i \leq j \leq k\}$ and $r_1=1$.  If $1\leq \ell <n$ then we
  define \( r(\ell,n) =(3r_{n-1}+7)\max\{r(i,n)\colon 1\leq i<
  \ell\}\) with the convention that $\max{\emptyset}=1$.

  We only need to consider the maps $\tv_{\ell n}$ with $\ell \in
  L_n=\{i: U_i\cap U_n \neq \emptyset\}$ and $\ell < n$.
  We proceed as in proof of (1) by induction on $n$ and $k\in L_n$
  with the additional provision that
  \begin{itemize}
  \item[$(3')$] $\tv_{ij}(\tu_i\cap \tu_j)\subset
    B(\baryv{ij},r(i,j)\ep)$, for all $ 1\leq i \leq j \leq n-1$ and
    for all $(i,j)$ with $i< k$, $i\in L_n$ and $j=n$.
  \end{itemize}

  The basic idea of the proof is to observe that it follows from the
  equations ($0'$), ($1'$) and ($2'$) that $v'_{\ell\,n}$ is close to
  $\tv_{\ell\,i}\tv_{in}$ (if $i < \ell$) or $\tv_{\ell\,i}v_{in}$ (if
  $\ell \leq i$) both of which are near $\baryv{\ell\,i}\baryv{in}$
  and hence $v'_{\ell,n}$ is close to $\baryv{\ell n}$.  It will
  follow that the image of $v'_{\ell n}$ is contained in a ball
  $B(\baryv{\ell n},r(\ell,n)\ep)$ where $r(\ell,n)$ is a universal
  constant computed recursively from previously determined
  $r(i,j)$. Therefore we can invoke the strengthened version of
  Tietze's theorem of \cite{Dugundji51} to extend $v'_{\ell,n}$ to a
  continuous map $\tv_{\ell,n}$ with values in convex open ball
  $B(\baryv{\ell n},r(\ell,n)\ep)$.

  Fix $k\in L_n$, $k<n$. By the inductive hypothesis, suppose that we
  have constructed $\tv_{ij}$ and they satisfy ($3'$).  We need to
  consider two cases.  The first is the case when $k=\min
  L_n$. Letting $\ell=k$ and $0<\varepsilon <1$, then from condition
  ($0'$), for each $x \in \tu_{\ell}\cap U_{j}\cap U_{n}$ with
  $\ell\leq j<n$, $j\in L_n:$
  \begin{align*}
    \|v'_{\ell\,n}(x)-\baryv{\ell\,n}\|
    &\leq
    \| \tv_{\ell\,j}(x)-\baryv{\ell\,j}\|
    \|v_{jn}(x)\|+\|\baryv{\ell\, j}\|
    \|v_{jn}(x)-\baryv{in}\| +\mbox{}\\
    &\quad +
    \| \baryv{\ell\, j} - \baryv{jn} \baryv{\ell\, n} \|\\
    &<
    r(\ell,i)\ep(1+\ep) + \ep(1+\ep) + \ep(3+2\ep)\\
    &\leq
    (3r_{n-1} + 7)\ep \leq  r(\ell,n)\ep.
  \end{align*}

  Let $\ell$ be the successor of $k$ in $L_n$. We may assume that
  $\ell<n$ otherwise we are done.  If $x\in \tu_{\ell}\cap U_{j}\cap
  U_{n}$ with $\ell\leq j<n$, $j\in L_n,$ then using ($2'$) it follows
  just as above that $\|v'_{\ell\,n}(x)-\baryv{\ell\,n}\|\leq
  (3r_{n-1}+7)\ep=r(\ell,n)\ep.$ On the other hand, if
  $x\in\tu_{\ell}\cap \tu_{i}\cap U_{n}$ with $ i<\ell $, $i\in L_n,$
  then using ($1'$) we have
  \begin{align*}
    \| v'_{\ell\,n}(x) - \baryv{\ell\,n} \|
    &\leq
    \| \tv_{\ell\,i}(x) - \baryv{\ell\,i} \|
    \| \tv_{in}(x)\| + \|\baryv{\ell\, i} \|
    \| \tv_{in}(x) - \baryv{in} \|
    + \|\baryv{\ell\, i} - \baryv{in}\baryv{\ell\, n} \|\\
    &<
    r(\ell,i)\ep(1 + ( r(i,n)+1 )\ep)
    + ( 1+\ep )r(i,n)\ep + \ep( 2+3\ep )\\
    &\leq
    (3r(\ell,i) + 7)r(i,n)\\
    &\leq
    (3r_{n-1}+7)\max\{r(i,n)\colon i< \ell, i\in L_n\}\\
    &\leq  r(\ell,n)\ep.
  \end{align*}
  In view of this estimates we can extend $v'_{\ell n}$ to $\tv_{\ell
    n}$ using the strengthened version of Tietze's theorem so that
  $\tv_{\ell n}(\tu_{\ell}\cap \tu_n)\subset B(\baryv{\ell
    n},r(\ell,n)\ep)$.  It follows that
  $\|\tv_{ij}(x)-\baryv{ij}\|<r_n \ep$ for all $x\in \tu_i\cap \tu_j$.
  This completes the proof.
\end{proof}

\section{Almost flat K-theory classes and the K-theoretical
  MF-property}
\label{sec:an-appl}

One of the motivations for this paper is the detection of nontrivial
K-theory elements of a group C*-algebra, via lifting of homomorphisms
$K_0(C^*(\Gamma))\to\Z$ to quasi-representations $C^*(\Gamma)\to
M_m(\C)$.  Suppose that the full assembly map is a bijection for a
discrete group $\Gamma$.  Roughly speaking, our main result states
that the quasi-representations $C^*(\Gamma)\to M_m(\C)$ which induce
interesting partial maps on K-theory are as abundant as the
non-trivial almost flat K-theory classes of the classifying space
$B\Gamma$.  More generally, for a C*-algebra $B$ we consider the
connection between almost flat K-theory classes in
$K_0(C(B\Gamma)\otimes B)$ and quasi-representations $C^*(\Gamma)\to
M_m(B)$ that implement a given homomorphism $K_*(C^*(\Gamma))\to
K_*(B)$.

Let $A$ be a unital C*-algebra.  A quasi-representation $\pi\colon
\Gamma\to \GL(A)$ extends to a unital linear contraction $\pi\colon
\ell^1(\Gamma) \to A$ in the obvious way.  We like to think of $\pi$
as ``inducing'' a partially defined map $\pi_\sharp\colon K_0(\ell^1
(\Gamma))\to K_0(A)$ (cf. \cite{Dadarlat12,Dadarlat14}).  We briefly
recall the definition of $\pi_\sharp$.  In the definition we write
$\chi$ for the function \( \zeta\mapsto \frac{1}{2\pi i} \int_C (z -
\zeta)^{-1} dz, \) where $C = \{ z\in \C : |z-1| = 1/4 \}$.

\begin{definition}[c.f. \cite{Dadarlat12}]\label{def:pushforward}
  Let $D$, $B$ be Banach algebras and let $\pi\colon D\to B$ be a
  unital contractive map.  Let $p\in M_m(\C)\otimes D$ be an
  idempotent and let $x= (\id_m\otimes \pi) (p) \in M_m(\C)\otimes B$.
  Define \( \pi_\sharp(p) = [ \chi(x) ] \in K_0(B) \) whenever $\| x^2
  - x\| < 1/4$. In a similar manner, one defines the pushforward
  $\pi_\sharp(u)\in K_1(B)$ of an invertible element $u\in
  M_m(\C)\otimes D$ as the class of $[(\id_m\otimes \pi) (u)],$ under
  the assumption that $\pi$ is (sufficiently) approximately
  multiplicative on a suitable finite subset of $D$ that depends on
  $u$.
\end{definition}

In general $\pi_\sharp(p)$ is not necessarily equal to $\pi_\sharp(q)$
if $[p]=[q]$ in $K_0(D)$.  To bypass this nuisance, we use a discrete
version of the asymptotic homomorphisms of Connes and Higson.

A discrete asymptotic homomorphism from an involutive Banach algebra
$D$ to C*-algebras $B_n$ consists of a sequence $\{\pi_n\colon D \to
B_n\}_{n=1}^\infty$ of maps such that
\[
\lim_{n \to \infty}
\left\{
  \begin{array}{l}
    \|\pi_n(a+\lambda a')-\pi_n(a)-\lambda\pi_n(a')\|\\
    \|\pi_n(a^*)-\pi_n(a)^*\|\\
    \|\pi_n(aa')-\pi_n(a)\pi_n(a')\|
  \end{array}
\right\}
=0
\]
for all $a,a'\in D$ and $\lambda\in \mathbb{C}$. The sequence
$\{\pi_n\}_{n}$ induces a $*$-homomorphism \( D \to \prod_{n=1}^\infty
B_n/\sum_{n=1}^\infty B_n.  \) If each $B_n$ is a matrix algebra over
some fixed C*-algebra $B$, then this further induces a group
homomorphism \[ K_*(D)\to \prod_{n=1}^\infty K_*(B)/\sum_{n=1}^\infty
K_*(B).  \] A discrete asymptotic homomorphism gives a canonical way
to push forward an element $x\in K_*(D)$ to a sequence
$(\pi_{n\,\sharp}(x))$ of elements of $K_*(B)$, which is well-defined
up to tail equivalence: two sequences are tail equivalent, written
$(y_n)\equiv (z_n)$, if there is $m$ such that $y_n=z_n$ for all
$n\geq m$.  Note that one can adapt Definition~\ref{def:pushforward}
to maps which are approximately contractive (in addition to being
approximately multiplicative).

\begin{remark}
  \label{rem:disc-asymp-obs}
  Let $\Gamma$ be a discrete countable group with a finite set of
  generators $\mathcal{F}$. We need the following observations:

  \begin{enumerate}
  \item A sequence of $(\mathcal{F},\delta_n)$-representations
    $\{\pi_n\colon \Gamma \to U(B_n)\}_{n=1}^\infty$, with $\delta_n
    \to 0$ as $n\to \infty$, induces a discrete asymptotic
    homomorphism (still written $(\pi_n)_{n=1}^\infty$) from the
    involutive Banach algebra $\ell^1(\Gamma)$ to the C*-algebras
    $B_n$.
  \item A discrete asymptotic homomorphism $\{\pi_n\colon
    \ell^1(\Gamma) \to B_n\}_{n=1}^\infty$ as above induces a
    $*$-homomorphism $\pi_\infty \colon \ell^1(\Gamma) \to
    B_\infty:=\prod_{n=1}^\infty B_n/\sum_{n=1}^\infty B_n$ and hence
    a $*$-homomorphism $\bar{\pi}_\infty \colon C^*(\Gamma) \to
    B_\infty$ such as the following diagram is commutative.
    \[
    \xymatrix{
      \ell^1(\Gamma)\ar[r]^-{\pi_\infty} \ar[d]_{j} & B_\infty  \\
      C^*(\Gamma)\ar[ur]_-{\bar{\pi}_\infty }& }
    \]
    (where $j$ is the canonical map).
  \item Let $\{\bar{\pi}_n\colon C^*(\Gamma) \to
    B_n\}_{n=1}^\infty$ be a discrete asymptotic homomorphism given by
    some set-theoretic lift of $\bar{\pi}_\infty$. If $y\in
    K_*(\ell^1(\Gamma))$, then we have the following tail equivalence:
    \begin{displaymath}
      \big( \bar{\pi}_{n\sharp}(j_*(y)) \big)_{n=1}^\infty
      \equiv
      \big( {\pi}_{n\sharp}(y) \big)_{n=1}^\infty
    \end{displaymath}
  \end{enumerate}
\end{remark}

The results of \cite{Dadarlat12} will allow us to relate the
push-forward of elements in the image of the full Baum-Connes assembly
map with the almost flat bundles we have constructed.  The relationshp
involves the Mishchenko line bundle and its push-forward, which we now
discuss.

Let $\widetilde{X}$ be the universal cover of $X=|\Lambda|$. Consider
the dual cover $\mathcal{C}_{\Lambda}=\{c_i\}_i$ of $X$ and the
associated open cover $\mathcal{V}_{\nu}=\{V_i\}_i$ where $V_i=\{x\in
X\colon d(x,c_i)<\nu\}$.  The Mishchenko line bundle is the bundle
$\widetilde{X}\times_{\Gamma} \ell^1(\Gamma)\to X$, obtained from
$\widetilde{X}\times \ell^1(\Gamma)$ by passing to the quotient with
respect to the diagonal action of $\Gamma$.  It is isomorphic to the
bundle $E$ obtained from the disjoint union $\bigsqcup V_i \times
\ell^1(\Gamma)$ by identifying $(x, a)$ with $(x, \gamma_{ij}a)$
whenever $x\in V_i\cap V_j$, where $\gamma_{ij}\in \Gamma$ are as in
Notation~\ref{not:edge-path-grp}; see for example \cite[Lemma
3.3]{Carrion-Dadarlat13}.  Let $\{\chi_i\}$ be a partition of unity
subordinate to $\{V_i\}$.  It follows that the Mishchenko line bundle
corresponds to the class of the projection
\begin{equation}\label{eqn:Mish-proj}
  e := \sum_{i,j} e_{ij}\otimes \chi_i^{1/2}\chi_j^{1/2}\otimes \gamma_{ij}
  \in M_N(\C)\otimes C(X)\otimes \C[\Gamma],
\end{equation}
where $\{e_{ij}\}$ are the canonical matrix units of $M_N(\C)$ and $N$
is the number of vertices in $\Lambda$.  We have inclusions of rings
$\C[\Gamma]\subset \ell^1(\Gamma)\subset C^*(\Gamma)$. The class of
the idempotent $e$ in $K_0(C(X)\otimes \ell^1(\Gamma))$ or
$K_0(C(X)\otimes C^*(\Gamma))$ is denoted by $\ell$.

\begin{notation}
  \label{not:e-pi}
  For an $(\mathcal{F}_\Lambda, \varepsilon)$-representation
  $\pi\colon \Gamma\to U(A)$ as in Definition~\ref{def:q-rep}, we set
  \begin{displaymath}
    \ell_\pi := (\id_{C(X)}\otimes \pi)_\sharp(e) \in K_0(C(X)\otimes A).
  \end{displaymath}

  From Definition~\ref{def:cocycle} we get the coordinate bundle
  $\beta(\pi)$ associated with $\pi$.  Applying
  Proposition~\ref{prop:bundle_extensions} to $\beta(\pi)$ we obtain
  an almost flat cocycle $\{v_{ij} : V_i\cap V_j \to \GL(A)\}$.  Let
  $E_\pi$ be the bundle constructed from the disjoint union $\bigsqcup
  V_{i} \times A$ by identifying $(x, a)$ with $(x, v_{ij}(x)a)$ for
  $x$ in $V_{ij}$.
\end{notation}

\begin{proposition}
  \label{prop:push-Mish-flat}
  There is $\varepsilon_0$ such that, for any
  $0<\varepsilon<\varepsilon_0$ and any $(\mathcal{F}_\Lambda,
  \varepsilon)$-representation $\pi\colon \Gamma\to U(A)$,
  \[
  [E_\pi]= \ell_\pi 
  \in K_0(C(X)\otimes A).
  \]
\end{proposition}

\begin{proof}
  Let $\varepsilon > 0$ and let $\pi\colon \Gamma\to U(A)$ be an
  $(\mathcal{F}_\Lambda, \varepsilon)$-representation.  Because the
  bundle $E_\pi$ is represented by the idempotent $p=\sum_{i,j}
  e_{ij}\otimes \chi_i^{1/2}\chi_j^{1/2}\otimes v_{ij}$, it follows
  that
  \[
  p-(\id_{C(X)}\otimes \pi)(e)=\sum_{i,j} e_{ij}\otimes
  \chi_i^{1/2}\chi_j^{1/2}\otimes (\pi(\gamma_{ij})-v_{ij}).
  \]
  Now, by the construction of $\{v_{ij}\}$, there is a constant $C$
  depending only on $\Lambda$ such that $ \sup_{x \in V_i\cap
    V_j}\|\pi(\gamma_{ij})-v_{ij}(x)\|<C \varepsilon$.  Therefore,
  $\|p-(\id_{C(X)}\otimes \pi)(e)\|<1/4$ if $\varepsilon_0$ is chosen
  to be sufficiently small.
\end{proof}

Let $X$ be a compact connected space and let $B$ be a unital
C*-algebra.  We consider locally trivial bundles $E$ over $X$ with
fiber finitely generated projective Hilbert-modules $F$ over $B$ and
structure group $\GL(A)$, where $A=L_B(F)$, the C*-algebra of
$B$-linear adjointable endomorphisms of $F$.  The K-theory group $
K_0(C(X)\otimes B)$ consists of formal differences of isomorphism
classes of such bundles.  Let $\mathcal{V}$ be a finite open cover of
$X$.  A bundle $E$ as above is $(\mathcal{V},\varepsilon)$-flat if it
admits an $(\mathcal{V},\varepsilon)$-flat associated
$\GL(A)$-principal bundle in the sense of
Definition~\ref{definition:to_alm_flat}(2) .

\begin{definition}
  An element $x \in K_0(C(X)\otimes B)$ is \emph{almost flat} if there
  is a finite open cover $\mathcal{V}$ of $X$ such that for every
  $\varepsilon > 0$ there are $(\mathcal{V},\varepsilon)$-flat bundles
  $E^\pm$ over $X$ such that $\alpha = [E^+] - [E^-]$. We say that $x
  \in K_0(C(X)\otimes B)$ is \emph{almost flat modulo torsion} if
  there is a torsion element $t\in K_0(C(X)\otimes B)$ such that $x-t$
  is almost flat.
\end{definition}

By the UCT given in \cite[Lemma 3.4]{Kasparov-Skandalis91}, the
Kasparov product
\[
KK(\mathbb{C},C(B\Gamma)\otimes B)\times
KK_*(C(B\Gamma),\mathbb{C})\to
KK_*(\mathbb{C},B),\quad
(x,z)\mapsto \langle x, z \rangle,
\]
induces an exact sequence
\begin{equation}
  \label{UCT}
  \Ext(K_*(B\Gamma),K_{*+1}(B))
  \rightarrowtail
  K_0(C(B\Gamma)\otimes B)
  \twoheadrightarrow
  \Hom(K_*(B\Gamma),K_*(B)).
\end{equation}
If $K_*(B)$ is finitely generated and torsion free, then the torsion
subgroup of $K_0(C(B\Gamma)\otimes B)$ coincides with the image of
$\Ext(K_*(B\Gamma),K_{*+1}(B))$.
\begin{theorem}
  \label{thm:almost-flat-surj}
  Let $\Gamma$ be a discrete countable group whose classifying space
  $B\Gamma$ is a finite simplicial complex and let $B$ be a unital
  C*-algebra.  Consider the following conditions:
  \begin{enumerate}
  \item For any $x\in K_0(C(B\Gamma)\otimes B)$ there is $t\in
    \Ext(K_*(B\Gamma),K_{*+1}(B))$ such that $x-t$ is almost
    flat.
  \item For any group homomorphism $h\colon K_*( C^*(\Gamma) )\to
    K_*(B)$ there exist discrete asymptotic homomorphisms
    $\{\pi_n^\pm\colon C^*(\Gamma)\to M_{k(n)}(B)\}_n$ such that \(
    (\pi^+_{n\,\sharp}(y) - \pi^-_{n\,\sharp}(y))\equiv (h(y)) \) for
    every $y$ in the image of the full assembly map $\mu\colon
    K_*(B\Gamma)\to K_*(C^*(\Gamma)).$
  \end{enumerate}
  Then (1) $\Rightarrow$ (2). Moreover if $K_*(B)$ is finitely
  generated and if $\mu$ is split injective, then (2) $\Rightarrow$
  (1).
\end{theorem}

\begin{proof}
  (1) $\Rightarrow$ (2).  Let $h\colon K_*( C^*(\Gamma) )\to K_*(B)$
  be given. Then $h\circ \mu \in \Hom(K_*(B\Gamma),K_*(B))$.  By the
  UCT \eqref{UCT}, there is $x\in K_0(C(B\Gamma)\otimes B)$ such that
  \[
  h\big( \mu(z) \big) = \langle x, z \rangle
  \quad \text{ for all }\quad z\in K_*(B\Gamma).
  \]
  Note that if $t\in \Ext(K_*(B\Gamma),K_{*+1}(B))$, then $\langle
  x+t, z \rangle=\langle x, z \rangle$.  Thus without any loss of
  generality we may assume that $x$ is almost flat. Therefore there
  exist a finite open cover $\mathcal{V}$ of $B\Gamma$, a decreasing
  sequence $(\ep_n)$ of positive numbers converging to $0$ and two
  sequences $(E_n^\pm)$ of bundles over $B\Gamma$ such that $E_n^\pm$
  are $(\mathcal{V},\varepsilon_n)$-flat and satisfies $x = [E_n^+] -
  [E_n^-]$ for all $n$. By passing to barycentric subdivisions of the
  simplicial structure $\Lambda$ of $B\Gamma$ we may assume that the
  dual cover $\mathcal{C}_{\Lambda}$ refines the open cover
  $\mathcal{V}$.  By Proposition~\ref{prop:T-normalized} we may
  arrange that the coordinate bundles underlying the $(E_n^\pm)$ are
  normalized.

  Write $F_n^\pm$ for the fibers of $E_n^\pm$; these are finitely
  generated projective Hilbert $B$-modules and therefore embed as
  direct summands of some $B^{k(n)}$.  This gives full-corner
  embeddings $A_n^\pm:=L_B(F_n^\pm)\subset M_{k(n)}(B)$.  Using
  Proposition~\ref{prop:qrep-existence} and
  Proposition~\ref{prop:perturb-qrep} we associate with $E_n^\pm$
  quasi-representations $\pi_n^\pm\colon \Gamma\to \U(A_n^\pm)$ such
  that $\lim_{n\to \infty}
  \|\pi_n^\pm(st)-\pi_n^\pm(s)\pi_n^\pm(t)\|=0$ for all $s,t \in
  \Gamma$ and $\lim_{n\to \infty}
  \|\pi_n^\pm(s^{-1})-\pi^\pm_n(s)^*\|=0$ for all $s\in \Gamma$.  The
  sequences $(\pi_n^\pm)$ induce morphisms of groups $\Gamma\to
  U(A^\pm_\infty)$ and hence $*$-homomorphisms $\pi_\infty^\pm\colon
  \ell^1(\Gamma) \to A_\infty^\pm$ and $\bar{\pi}_\infty^\pm\colon
  C^*(\Gamma) \to A_\infty^\pm$ where $A_\infty^\pm=\prod_{n=1}^\infty
  A_n^\pm /\sum_{n=1}^\infty A_n^\pm$. Let $\bar{\pi}^\pm\colon
  C^*(\Gamma) \to \prod_{n=1}^\infty A_n^\pm$ be a set-theoretic
  lifting of $\pi_\infty^\pm$. Write
  $\bar{\pi}^\pm=(\bar{\pi}_n^\pm)_n$.  For a sufficiently
  multiplicative quasi-representation $\pi\colon \Gamma \to
  \U(A)_\delta$ and a sufficiently small $\delta>0$, we will denote by
  $E_\pi$ the corresponding almost flat bundle constructed using the
  cocycle $\beta(\pi)$ constructed in
  Proposition~\ref{prop:almost-flat-existence} (see
  Notation~\ref{not:e-pi}).  For $n$ sufficiently large we have that
  $[E_{\pi_n^\pm}] = [E_n^\pm]$. This follows from
  Theorem~\ref{thm:approx-inverse} and
  Proposition~\ref{prop:correspondence_af} since bundles whose
  cocycles are sufficiently close to each other are isomorphic.

  Let us recall that the full assembly map $\mu\colon K_*(B\Gamma)\to
  K_*(C^*(\Gamma))$ is implemented by the Mishchenko line bundle $\ell
  \in K_0( C(B\Gamma)\otimes C^*(\Gamma ))$, via the Kasparov product
  \[
  KK(\C,C(B\Gamma)\otimes C^*(\Gamma ))\times KK_*(C(B\Gamma),\C)\to
  KK_*(\C,C^*(\Gamma)),
  \]
  $(\ell,z)\mapsto \mu(z):=\langle \ell,z
  \rangle.$

  We have seen earlier \eqref{eqn:Mish-proj} that one can represent
  $\ell$ by a projection $e$ in matrices over $ C(B\Gamma)\otimes
  \mathbb{C}[\Gamma]$.  So long as $n$ is sufficiently large,
  Proposition~\ref{prop:push-Mish-flat} guarantees that
  $[E_{\pi_n^\pm}]$ equals $\ell_{\pi_n^\pm}$, the push-forward of $e$
  by $\mathrm{id}_{C(B\Gamma)}\otimes \pi_n^\pm$ in $K_0(
  C(B\Gamma)\otimes A_n^\pm)\cong K_0( C(B\Gamma)\otimes B)$. The
  latter isomorphism is induced by the full-corner embeddings
  $A_n^\pm\subset M_{k(n)}(B)$.  It follows that
  \[
  \langle x, z \rangle
  =
  \langle [E_n^+], z \rangle - \langle [E_n^-], z \rangle
  =
  \langle \ell_{\pi_{n}^{+}}, z \rangle
  - \langle \ell_{\pi_{n}^{-}}, z \rangle
  \]
  for all $z\in K_*(B\Gamma)$.  Let $\mu_{\ell^1}\colon
  K_*(B\Gamma)\to K_*(\ell^1(\Gamma))$ be Lafforgue's $\ell^1$-version
  of the assembly map.  It is known that $j_*\circ \mu_{\ell^1}=\mu$
  where $j\colon \ell^1(\Gamma)\to C^*(\Gamma)$ is the canonical map
  \cite{Lafforgue02}.  By \cite[Theorem~3.2 and Corollary
  3.5]{Dadarlat12} we have that
  \[
  \langle \ell_{\pi_{n}^{\pm}}, z \rangle
  =
  {\pi_n^\pm}_\sharp\big( \mu_{\ell^1}(z) \big)
  \]
  for each $z\in K^*(B\Gamma)$ so long as $n$ is sufficiently large.
  (For $z\in K_0(B\Gamma)$ we interpret ${\pi_n^\pm}_\sharp\big(
  \mu_{\ell^1}(z) \big)$ as
  ${\pi_n^\pm}_\sharp(p_z)-{\pi_n^\pm}_\sharp(q_z)$ where $p_z,q_z$
  are projections with $\mu_{\ell^1}(z)=[p_z]-[q_z]$. There is a
  similar interpretation of ${\pi_n^\pm}_\sharp\big( \mu_{\ell^1}(z)
  \big)$ for $z\in K_1(B\Gamma)$ obtained by replacing idempotents by
  invertible elements.)  Therefore,
  \[
  \big(
  \pi^+_{n\,\sharp}\big( \mu_{\ell^1}(z) \big)
  - \pi^-_{n\,\sharp}\big( \mu_{\ell^1}(z) \big)
  \big)
  \equiv
  (\langle x, z \rangle)
  =
  \big(h( \mu(z) )\big)
  \]
  for all $z\in K_*(B\Gamma)$.  From
  Remark~\ref{rem:disc-asymp-obs}(3) we deduce that
  \[
  \bar{\pi}^\pm_{n\,\sharp}\big( \mu(z) \big)
  =
  \bar{\pi}^\pm_{n\,\sharp}\big(j_*(\mu_{\ell^1}(z)) \big)
  \equiv
  {\pi}^\pm_{n\,\sharp}\big(\mu_{\ell^1}(z) \big)
  \]
  \[
  \big(
  \bar{\pi}^+_{n\,\sharp}\big( \mu(z) \big)
  - \bar{\pi}^-_{n\,\sharp}\big( \mu(z) \big)
  \big)
  \equiv
  \big(h( \mu(z) )\big)
  \]
  for all $z\in K_*(B\Gamma)$.  The discrete asymptotic homomorphisms
  $\{\bar{\pi}_n^\pm:C^*(\Gamma)\to M_{k(n)}(B)\}_n$ have the desired
  properties.  \vskip 4pt

  (2) $\Rightarrow$ (1).  Let us assume now that $K_*(B)$ is finitely
  generated and that $\mu$ is split injective.  Let $x\in
  K_0(C(B\Gamma)\otimes B)$ be given. We will find an almost flat
  element $y\in K_0(C(B\Gamma)\otimes B)$ such that $x-y\in
  \Ext(K_*(B\Gamma),K_{*+1}(B))$.  Since $\mu$ is split-injective by
  hypothesis (i.e. the image of $\mu$ is a direct summand of
  $K_*(C^*(\Gamma)),$ it follows from the exactness of the sequence
  \eqref{UCT} that there is a homomorphism $h\colon
  K_*(C^*(\Gamma))\to K_*(B)$ such that $h\circ \mu (z)=\langle x,z
  \rangle$ for all $z\in K_*(B\Gamma)$.  By the assumptions in (2)
  there are two discrete asymptotic homomorphisms $\{\pi_n^\pm\colon
  C^*(\Gamma)\to M_{k(n)}(B)\}_n$ such that
  \begin{equation}
    \label{eq:10}
    (
    \pi^+_{n\,\sharp}\big( \mu(z) \big)
    - \pi^-_{n\,\sharp}\big( \mu(z) \big)
    )
    \equiv
    \big(h\big( \mu(z) \big)\big)
    =
    (\langle x, z \rangle)
  \end{equation}
  for all $z\in K^0(B\Gamma)$.  By a standard perturbation argument we
  may assume that $\pi_n^\pm(s)\in U_{k(n)}(B)$ for all $n$ and $s\in
  \Gamma$.  Invoking \cite[Thm.~3.2 and Cor. 3.5]{Dadarlat12} we
  obtain that
  \[
  (\langle \ell_{\pi_{n}^{\pm}}, z \rangle)
  \equiv
  \big({\pi_n^\pm}_\sharp\big( \mu(z) \big)\big).
  \]
  By Proposition~\ref{prop:push-Mish-flat}, we have that
  $\ell_{\pi_{n}^{\pm}} = [E_{\pi_n^\pm}]$ where the bundles
  $E_{\pi_n^\pm}$ are $(\mathcal{V},\varepsilon_n)$-flat and
  $\varepsilon_n\to 0$.  If we set $x_n:=[E_{\pi_n^+}]-[E_{\pi_n^-}]$,
  it follows from ~\eqref{eq:10} that for any $z\in K_*(B\Gamma)$,
  there is $n_z$ such that $\langle x,z \rangle=\langle x_n,z \rangle$
  for $n\geq n_z$.  Since $K_*(B\Gamma)$ is finitely generated there
  exists $n_0$ such that $x-x_n \in H:=\Ext(K_*(B\Gamma),K_{*+1}(B))$
  for all $n\geq 0$.  Since $K_*(B)$ is finitely generated, the group
  $H$ is finite. Therefore after passing to a subsequence of $(x_n)$
  we may arrange that the sequence $(x-x_n)$ is constant and so there
  is $t\in H$ such that $x+t=x_n$ for all $n$. It follows that
  $y:=x+t$ is is almost flat and $x-y \in H$.
\end{proof}

If $B\Gamma$ is a finite complex and the group $K_*(B)$ is finitely
generated and torsion-free, it follows that the group
$\Ext(K_*(B\Gamma),K_{*+1}(B))$ is finite and by \eqref{UCT} it
coincides with the torsion subgroup of $K_0(C(B\Gamma)\otimes B)$.  By
taking $B=\C$ in Theorem~\ref{thm:almost-flat-surj} we obtain the
following.

\begin{theorem}
  \label{thm_cor:almost-flat-surj}
  Let $\Gamma$ be a discrete countable group whose classifying space
  $B\Gamma$ is a finite simplicial complex. If the full assembly map
  is bijective, then the following conditions are equivalent
  \begin{enumerate}
  \item All elements of $K^0(B\Gamma)$ are almost flat modulo
    torsion.
  \item For any group homomorphism $h\colon K_0( C^*(\Gamma) )\to \Z$
    there exist discrete asymptotic homomorphisms $\{\pi_n^\pm\colon
    C^*(\Gamma)\to M_{k(n)}(\C)\}_n$ such that \(
    (\pi^+_{n\,\sharp}(y) - \pi^-_{n\,\sharp}(y))\equiv (h(y)) \) for
    every $y \in K_0(C^*(\Gamma)).$
  \end{enumerate}
\end{theorem}

With the terminology from the introduction, condition (2) above
amounts to saying that $C^*(\Gamma)$ is K-theoretically MF.

\begin{remark}
  Gromov indicates in \cite{Gromov96, Gromov95} how one constructs
  nontrivial almost flat $K$-theory classes for residually finite groups
  $\Gamma$ that are fundamental groups of even dimensional
  non-positively curved compact manifolds.
\end{remark}


\bibliographystyle{amsplain}

\providecommand{\bysame}{\leavevmode\hbox to3em{\hrulefill}\thinspace}
\providecommand{\MR}{\relax\ifhmode\unskip\space\fi MR }
\providecommand{\MRhref}[2]{%
  \href{http://www.ams.org/mathscinet-getitem?mr=#1}{#2}
}
\providecommand{\href}[2]{#2}

\end{document}